\documentclass[11pt,reqno,a4paper]{amsart}
\usepackage[utf8]{inputenc}
\usepackage{amsmath,amssymb,url,eucal}
\usepackage{supertabular}
\usepackage{mathrsfs}
\usepackage{fullpage}
\usepackage{epsfig}

\newcommand{\CC}{{\mathbf C}}
\newcommand{\RR}{{\mathbf R}}
\newcommand{\QQ}{{\mathbf{Q}}}

\newcommand{\ZZ}{{\mathbf{Z}}}

 \DeclareMathOperator{\Prob}{\mathbf{P}}   
 \DeclareMathOperator{\E}{\mathbf{E}}      
 \DeclareMathOperator{\vol}{Vol}           
 \DeclareMathOperator{\Arg}{Arg}           

\def\sums{\mathop{\sum \Bigl.^{*}}\limits}

\newcommand{\convlaw}{\overset{\mbox{\rm \scriptsize law}}{\Rightarrow}}

\def\stacksum#1#2{{\stackrel{{\scriptstyle #1}}
{{\scriptstyle #2}}}}

\newcommand{\demi}{{\textstyle{\frac{1}{2}}}}

\newcommand{\hyperg}[4]{\phantom{}_2F_{1}({{#1}},{{#2}};{{#3}};{{#4}})}

\newcommand{\eps}{\varepsilon}
\newcommand{\Cc}{\CC}

\newcommand{\Zz}{\ZZ}

\newcommand{\Rr}{\RR}

\newcommand{\Qq}{\QQ}
\newcommand{\Fp}{\mathbf{F}}

\newcommand{\proba}{\Prob}
\newcommand{\expect}{\E}

\newcommand{\mods}[1]{\,(\mathrm{mod}\,{#1})}

\newcommand{\ra}{\rightarrow}
\newcommand{\lra}{\longrightarrow}

\DeclareMathOperator{\Imag}{Im}
\DeclareMathOperator{\Reel}{Re}

\DeclareMathOperator{\Tr}{Tr}

\renewcommand{\leq}{\leqslant}
\renewcommand{\geq}{\geqslant}

\newtheorem{theorem}{Theorem}
\newtheorem{corollary}[theorem]{Corollary}
\newtheorem{lemma}[theorem]{Lemma}
\newtheorem{proposition}[theorem]{Proposition}

\newtheorem{conjecture}[theorem]{Conjecture}
\theoremstyle{definition}

\newtheorem{example}{Example}
\theoremstyle{remark}
\newtheorem{remark}{Remark}
\newtheorem*{rem}{Remark}

\begin{document}
\title[Mod-Gaussian convergence and $\zeta(1/2+it)$]
{Mod-Gaussian convergence and the value distribution of
  $\zeta(1/2+it)$ and related quantities}

 \author{E. Kowalski} 
\address{ETH Z\"urich -- D-MATH \\ R\"{a}mistrasse
   101\\ 8092 Z\"{u}rich, Switzerland}
 \email{kowalski@math.ethz.ch}

\author{A. Nikeghbali}
 \address{Institut f\"ur Mathematik,
 Universit\"at Z\"urich, Winterthurerstrasse 190,
 8057 Z\"urich,
 Switzerland}
 \email{ashkan.nikeghbali@math.uzh.ch}

\begin{abstract}
  In the context of mod-Gaussian convergence, as defined previously in
  our work with J. Jacod, we obtain asymptotic formulas and lower
  bounds for local probabilities for a sequence of random vectors
  which are approximately Gaussian in this sense, with increasing
  covariance matrix. This is motivated by the conjecture concerning
  the density of the set of values of the Riemann zeta function on the
  critical line. We obtain evidence for this fact, and derive
  unconditional results for random matrices in compact classical
  groups, as well as for certain families of $L$-functions over finite
  fields.
\end{abstract}



\maketitle

\section{Introduction}

It is well-known (see, e.g.,~\cite[Th. 11.9]{titchmarsh}) that, for
$1/2<\sigma<1$, the set of values $\zeta(\sigma+it)$, $t\in\Rr$, is
dense in the complex plane. In fact, much more is true: it was proved
by Bohr and Jessen that there exists a Borel probability measure
$\mu_{\sigma}$ on $\Cc$, such that the support of $\mu_{\sigma}$ is
the whole complex plane, and such that the convergence in law
$$
\frac{1}{2T}\int_{-T}^T{f( \log \zeta(\sigma+it))dt}
\ra \int_{\Cc}{f(z)d\mu_{\sigma}(z)},
$$
holds for $f\,:\, \Cc\ra \Cc$ continuous and bounded.
\par
The corresponding density question for $\sigma=1/2$ is, however, still
open (it was apparently first raised by Ramachandra during the 1979
Durham conference, but seems to appear in print only in Heath-Brown's
note in~\cite[11.13]{titchmarsh}): the difficulty is that the values
$\zeta(1/2+it)$, $|t|\leq T$, do not have a limiting distribution, as
evidenced already by the Hardy-Littlewood asymptotic
$$
\frac{1}{T}\int_0^T{|\zeta(1/2+it)|^2dt}\sim (\log
T),\quad\quad\text{ as } T\ra +\infty,
$$
or by Selberg's result that $\log |\zeta(1/2+it)|$, $|t|\leq T$, is
asymptotically normal with variance \emph{growing to infinity} (see
also the work of Ghosh~\cite{ghosh} for the imaginary part,
and~\cite[\S 5]{bombieri-hejhal} for a recent proof). In other words,
``most'' values of $\zeta(1/2+it)$ are rather large, though the zeta
function is zero increasingly often as the imaginary part grows.
\par
In this paper, we show (Corollary~\ref{cor-zeta}) how the density of
values of zeta on the critical line would follow rather directly from
a suitable version of the Keating-Snaith moment conjectures, which we
viewed in our previous work with J. Jacod~\cite{jkn} as a refined
version of the Gaussian model.  In fact, under suitable assumptions,
we could prove a quantitative result, bounding from above the smallest
$t\geq 0$ for which $\zeta(1/2+it)$ lies in a given open disc in
$\Cc$.  This argument is based on a very general probabilistic
estimate proved in Section~\ref{sec-proba}, which throws some light on
the nature of the mod-Gaussian convergence that we defined
in~\cite{jkn}. We hope that this result will be of further use. In
another paper (jointly with F. Delbaen, see~\cite{dkn}), it will be
seen that one can weaken considerably the assumption needed in order
to prove the density of values of $\zeta(1/2+it)$ (but without
quantitative information).
\par
As applications of the general result, we will also prove the
following theorems in Section~\ref{sec-examples} (the precise versions
are given there).

\begin{theorem}\label{th-rmt}
  Let $z_0\in \Cc^{\times}$ be arbitrary, $\eps>0$ such that $\eps\leq
  |z_0|$. There exists $N_0(z_0,\eps)$, which can be bounded
  explicitly, such that
\begin{equation}\label{eq-rmt-un}
  \mu_N(\{g\in U(N)\,\mid\, |\det(1-g)-z_0|<\eps\})
\gg \Bigl(\frac{\eps}{|z_0|}\Bigr)^2\frac{1}{\log N}
\end{equation}
provided $N\geq N_0$, where $\mu_N$ denotes probability Haar measure
on the unitary group $U(N)\subset GL(N,\Cc)$, and the implied
constant is absolute.
\end{theorem}

\begin{theorem}\label{th-eulerproduct}
Define 
\begin{equation}\label{eq-pnt}
  P_N(t)=\prod_{p\leq N}{(1-p^{-1/2-it})^{-1}}
\end{equation}
for $N\geq 1$ and $t\in\Rr$.  Let $z_0\in \Cc^{\times}$ be arbitrary,
$\eps>0$ such that $\eps\leq |z_0|$. There exists $N_0(z_0,\eps)$,
explicitly bounded, such that
$$
\liminf_{T\ra +\infty}{\frac{1}{T}\lambda(\{t\leq T\,\mid\, P_N(t)\in
  V\})}\gg \Bigl(\frac{\eps}{|z_0|}\Bigr)^2
\frac{1}{\log\log N},
$$
for all $N\geq N_0$, where $\lambda$ is the Lebesgue measure and the
implied constant is absolute.
\end{theorem}

In a different direction, we obtain some evidence for the density of
$\zeta(1/2+it)$ by looking at special values of families of
$L$-functions over finite fields. In doing so, we also consider the
analogue of Theorem~\ref{th-rmt} for symplectic and orthogonal
matrices. We refer to Section~\ref{sec-ff} for precise statements and
definitions, and only state here one appealing (qualitative)
corollary:

\begin{theorem}\label{th-ff}
  The set of central values of the $L$-functions attached to
  non-trivial primitive Dirichlet characters of $\Fp_p[X]$, where $p$
  ranges over primes, is dense in $\Cc$.
\end{theorem}

\par
\medskip
\par
\textbf{Notation.} As usual, $|X|$ denotes the cardinality of a set.
By $f\ll g$ for $x\in X$, or $f=O(g)$ for $x\in X$, where $X$ is an
arbitrary set on which $f$ is defined, we mean synonymously that there
exists a constant $C\geq 0$ such that $|f(x)|\leq Cg(x)$ for all $x\in
X$. The ``implied constant'' refers to any value of $C$ for which this
holds. It may depend on the set $X$, which is usually specified
explicitly, or clearly determined by the context. Similarly, $f\asymp
g$ for $x\in X$ means $f\ll g$ and $g\ll f$, both for $x\in X$.
We write $(x)_j=x(x+1)\cdots (x+j-1)$ the Pochhammer symbol.
\par
\medskip
\par
\textbf{Acknowledgments.} The first version of this paper was written
while the first author was on sabbatical leave at the Institute for
Advanced Study (Princeton, NJ); many thanks are due to this
institution for its support. This material is based upon work
supported by the National Science Foundation under agreement
No. DMS-0635607.
\par
The second author was partially supported by SNF Schweizerischer
Nationalfonds Projekte Nr. 200021 119970/1.
\par
Thanks to K. Soundararajan for pointing out that the density
conjecture for $\zeta(1/2+it)$ might be a good problem to study using
complex moments of the zeta function, and to R. Heath-Brown for
explaining us the history of the question. Many thanks also to N. Katz
for discussions and explanations surrounding issues of monodromy
computations. Many thanks also to F. Delbaen for useful discussions
concerning the underlying probabilistic framework.
\par
The graphs were produced using \textsc{Sage 4.2}~\cite{sage}, relying
on the Barnes function routines in the \textsc{mpmath} package.

\section{Mod-Gaussian convergence and local probabilities}
\label{sec-proba}

In this section, which is purely probabilistic, we present two
versions of ``local'' bounds for probabilities in the case of
sufficiently uniform mod-Gaussian convergence of sequences of random
vectors. This may be compared with the local central limit theorem
(see, e.g.,~\cite[\S 10.4]{breiman} for the one-dimensional case). In
fact, as F. Delbaen pointed out, one can obtain qualitative statements
that generalize both the standard local central limit theorem and
recover our results below under weaker assumptions. However, our
emphasis is on explicit quantitative lower bounds for local
probabilities.
\par
We first introduce the definition, generalizing~\cite{jkn} to random
vectors.  Fix some integer $m\geq 1$, and let $(X_N)$ be a sequence of
$\Rr^m$-valued random variables defined on a probability space
$(\Omega,\Sigma,\proba)$ (as is the case for convergence in law, we
could work without change with random variables defined on different
probability spaces).  Let
$$
Q_N(t)=Q_N(t_1,t_2,\ldots, t_m)
$$
be a sequence of non-negative quadratic forms on $\Rr^m$. The sequence
$(X_N)$ is then said to be convergent in the mod-Gaussian sense with
covariance $Q_N$ and limiting function $\Phi$ if
\begin{equation}\label{eq-converge}
\lim_{N\ra +\infty}{\exp(Q_N(t)/2)\expect(e^{it\cdot X_N})}=\Phi(t)
\end{equation}
locally uniformly for $t\in\Rr^m$; $\Phi$ is then a function
continuous at $0$ and $\Phi(0)=1$. Here, $\cdot$ denotes the standard
inner product on $\Rr^m$.
\par
The intuitive meaning is that, in some sense, $X_N$ is ``close'' to a
(centered) Gaussian vector $G_N$ with covariance matrix $Q_N$. As
in~\cite{jkn}, this notion is of most interest if the covariance
``goes to infinity''. However, in contrast with the case of $m=1$,
this can mean different things because there is more than a single
variance parameter involved.
\par
To discuss this, we diagonalize $Q_N$ in an orthonormal
basis\footnote{\ With respect to the standard inner product on
  $\Rr^m$.} in the form
$$
Q_N(t)=\delta_{1,N}u_1^2+\cdots+\delta_{m,N} u_m^2,\quad\quad
0\leq \delta_{1,N}\leq \delta_{2,N}\leq \cdots\leq \delta_{m,N}
$$
where $u=H_N(t)$ is the necessary (orthogonal) change of
variable. Then ``$Q_N$ goes to infinity,'' in the weakest sense, means
that the largest eigenvalue $\delta_{m,N}$ goes to $+\infty$ as $N\ra
+\infty$.
\par
We are interested in the distribution of values of $X_N$ as $N$ grows;
clearly, if (say) the $Q_N$ are already diagonalized in the canonical
basis and $\delta_{1,N}$ is constant, these values will have first
coordinate much less spread out than the last ones. To simplify our
discussion, and because this is the situation in our applications, we
will assume this behavior does not occur and that in fact the smallest
eigenvalue goes to infinity.  For simplicity, we will assume in fact
that for some fixed $\mu>0$, we have
\begin{equation}\label{eq-grow-not-too-much}
  \delta_{m,N}\leq \delta_{1,N}^{\mu},\quad\quad \delta_{m,N}\ra +\infty,
\end{equation}
so that also $\delta_{1,N}\ra +\infty$ (we say that the convergence is
\emph{balanced}). In our main applications, this will be the case with
$\mu=1$. Note in particular that it follows that the discriminant
$$
\sigma_N=\delta_{1,N}\cdots\delta_{m,N}\geq \delta_{1,N}^m
$$
goes to infinity as $N\ra +\infty$, and moreover
\begin{equation}\label{eq-mmu}
\sigma_N\leq \delta_{1,N}^{m\mu}.
\end{equation}
\par
Our question is now the following: given $(X_N)$, $(Q_N)$, as above,
with this type of mod-Gaussian convergence, can we bound from below
the probability
$$
\proba(X_N\in U),
$$
where $U$ is a fixed open set in $\Rr^m$? 
\par
Denoting by $\tilde{Q}_N(x)$ the dual quadratic form, the Gaussian
model suggests that, if $U$ is relatively compact (e.g., some
non-empty open ball), we could expect
\begin{equation}\label{eq-model}
\proba(X_N\in U)
\approx \proba(G_N\in U)=
\frac{1}{(2\pi)^{m/2}\sqrt{\sigma_N}}
\int_{U}{
e^{-\tilde{Q}_N(x)/2}dx
}
\sim 
\frac{\vol(U)}{(2\pi)^{m/2}\sqrt{\sigma_N}},
\end{equation}
as $N\ra +\infty$, since $\delta_{1,N}\ra +\infty$ implies that
$\tilde{Q}_N(x)\ra 0$ for all $x\in\Rr^m$. We will confirm that this
holds in certain conditions at least. We strive especially for lower
bounds on $\proba(X_N\in U)$, which we wish to be quantitative, so
that we can determine some $N_0$ (depending explicitly on $U$) for
which
$$
\proba(X_{N_0}\in U)>0.
$$
\par
Note that such a quantitative result must depend on the location of
the open set $U$, whereas the limit itself only depends on the volume,
as seen above.
\par
The specific hypothesis we use may seem somewhat arbitrary, but they
turn out to be satisfied (with room to spare) in the later
applications.

\begin{theorem}\label{th-2}
  Let $m\geq 1$ be fixed and let $(X_N)$ be a sequence of
  $\Rr^m$-valued random variables defined on $(\Omega,\Sigma,\proba)$,
  such that $(X_N)$ converges in the mod-Gaussian sense with
  covariance $(Q_N)$, and that the convergence is $\mu$-balanced with
  $\mu>0$, with $\sigma_N\geq 1$ for all $N$. Let $(G_N)$ be Gaussian
  random variables with covariance matrices given by $(Q_N)$, so that
$$
\exp(-Q_N(t)/2)=\expect(e^{it\cdot G_N}).
$$
\par
Assume moreover the following three conditions:
\par
\emph{(1)} There exist constants $a>0$, $\alpha>0$ and $C>0$ such
that, for any $N\geq 1$ and $t\in\Rr^m$ such that $\|t\|\leq
\sigma_N^a$, we have
\begin{equation}\label{eq-strong-convergence}
\expect(e^{it\cdot X_N})=\Phi(t)\exp(-Q_N(t)/2)
\Bigl\{
1+O\Bigl(\frac{1}{\exp(\alpha \sigma_N^C)}\Bigr)
\Bigr\}.
\end{equation}
\par
\emph{(2)} The function $\Phi$ is of class $C^1$ on $\{\|t\|<2\}$.
\par
\emph{(3)} For some $A\geq 1$ and $\beta\geq 0$, we have
\begin{equation}\label{eq-order-2}
|\Phi(t)|\ll \exp(\beta \|t\|^{A}),
\end{equation}
for $t\in \Rr^m$.
\par
Let $D>0$ be any number such that
\begin{equation}\label{eq-d}
  D>2(m+1+\max\{a^{-1},A/C, 3m(m+1)\mu A\}).
\end{equation}
\par
Then, for any fixed non-empty open box
$$
U=\{x\in \Rr^m \,\mid\, \|x-x_0\|_{\infty}<\eps\}\subset \Rr^m,
$$
with width $\eps$ such that $0<\eps\leq 1$, we have
\begin{equation}\label{eq-quantitative}
\proba(X_N\in U) =
\proba(G_N\in U)
+O\Bigl(\frac{1}{\sigma_N^{1/2+1/D}}+\frac{\eps^{-m}}{\sigma_N}\Bigr),
\end{equation}
for $N\geq 1$, where the implied constant depends only on $(m, \Phi,
a,\alpha,C)$ and the implied constant
in~\emph{(\ref{eq-strong-convergence})}.
\par
In particular, for any fixed non-empty open set $U\subset \Rr^m$, we
have
$$
\proba(X_N\in U)\gg \frac{1}{\sqrt{\sigma_N}}
$$
provided $N\geq N_0$, where $N_0$ and the implied constant depend on
$U$ and the same data as above.
\end{theorem}

Note the following elementary lower bound, valid if $\eps\leq 1$:
\begin{equation}\label{eq-lower-easy}
\proba(G_N\in U)\gg \frac{\eps^m}{\sqrt{\sigma_N}}
\exp\Bigl(-\frac{\tilde{Q}_N(x_0)}{2}\Bigr)
\end{equation}
where the implied constant depends only on $m$; this is where the
location of $U$ enters, since the error term will only be smaller than
this, roughly, when $\tilde{Q}_N(x_0)\asymp 1$.

\begin{remark}
The growth condition~(\ref{eq-order-2}) is in fact a consequence of
the uniform mod-Gaussian convergence, at least provided the sequence
$(\sigma_N)$ does not grow too fast. For instance, if
\begin{equation}\label{eq-growth-condition}
\sigma_{N+1}\leq M\sigma_N^B
\end{equation}
for all $N\geq 1$, for some constants $M\geq 1$, $B\geq 0$, we can
obtain~(\ref{eq-order-2}) with $A=2+B/a$. Indeed, we can write
$$
|\Phi(t)|=\Bigl|\Phi_N(t)e^{Q_N(t)/2}\Bigl(1+O\Bigl(
\frac{1}{\exp(\alpha \sigma_N^C)}\Bigr)\Bigr)\Bigr|\leq 2 e^{Q_N(t)/2}
$$
if $N$ is large enough and $\|t\|\leq \sigma_N^a$. Note then that
$$
Q_N(t)\leq \delta_{m,N}\|t\|^2\leq \delta_{1,N}^m\|t\|^2\leq
\sigma_N\|t\|^2.
$$
\par
We now fix $N\geq 1$ minimal such that
$$
\sigma_{N-1}\leq \|t\|^{1/a}\leq \sigma_N,
$$
and if this value of $N$ is large enough, we get $\sigma_N\leq
C\sigma_{N-1}^B\leq M\|t\|^{B/a}$, and hence
$$
|\Phi(t)|\leq 2\exp(\sigma_N\|t\|^2)\leq 2\exp(M\|t\|^{2+B/a}),
$$
as desired. On the other hand, if this chosen $N$ is too small,
$\|t\|$ is bounded, and the desired estimate is trivial.
\end{remark}

\begin{proof}[Proof of Theorem~\ref{th-2}]
  Let $\delta_N=\delta_{1,N}$ be the smallest eigenvalue of $Q_N$, so
  that $Q_N(t)\geq \delta_N\|t\|^2$ for all $t\in\Rr^m$ and $N\geq
  1$. For simplicity, we denote also
\begin{equation}
\label{eq-uniformity}
\gamma_N=\exp(\alpha\sigma_N^{C}),
\end{equation}
as in~(\ref{eq-strong-convergence}).
\par
We now first fix $w$ such that $0<w<1$, and then fix a smooth,
compactly supported function $g_0$ on $\Rr$ such that
\begin{gather*}
  0\leq g_0\leq 1,\\
  g_0(x)=0\text{ for } |x|\geq 1,\quad\quad
  g_0(x)=1\text{ for } |x|\leq w,\\
  |g_0^{(j)}(x)|\ll_j \Delta^j,\text{ for } j\geq 0,\ x\in\Rr,
\end{gather*}
where $\Delta=(1-w)^{-1}$ and the implied constant depends only on $j$
(we will define $w$ to be a function of $N$ at the end, and hence must
be careful to have estimates uniform in terms of $w$; this is provided
by using only the above properties of $g_0$; the maximal value of $j$
used will also be bounded only in terms of $m$, $\Phi$ and the data
in~(\ref{eq-strong-convergence})). It is classical that such a
function exists (examples are constructed in~\cite[\S
1.4]{hormander}). Then define
$$
f_0(x)=f_0(x_1,\ldots,x_m)=\prod_{1\leq j\leq m}{g_0(x_j)}
$$
for $x\in\Rr^m$. It follows that
\begin{gather*}
  0\leq f_0\leq 1,\\
  f_0(x)=0\text{ for } \|x\|_{\infty}\geq 1,\quad\quad
  f_0(x)=1\text{ for } \|x\|_{\infty}\leq w.
\end{gather*}
\par
Next, we define
$$
f(x)=f_0\Bigl(\frac{x-x_0}{\eps}\Bigr),
$$
and we start our argument with the obvious inequality
$$
\proba(X_N\in U)\geq \expect(f(X_N))
=\int_{\Rr^m}{f(x)d\nu_N(x)}
$$
where $\nu_N$ is the law of $X_N$. Applying the Plancherel formula, we
get
$$
\proba(X_N\in U)\geq \int_{\Rr^m}{f(x)d\nu_N(x)}=
\int_{\Rr^m}{\hat{f}(t)\Phi_N(t)dt}
$$
where $\Phi_N(t)=\expect(e^{it\cdot X_N})$ is the characteristic
function of $X_N$ and 
$$
\hat{f}(t)=\frac{1}{(2\pi)^{m/2}}\int_{\Rr^m}{f(x)e^{-it\cdot x}dx}
$$
denotes the Fourier transform of $f$ (the smoothness of $f$ guarantees
that $\hat{f}$ is in $L^1$, so the Plancherel formula is valid by a
simple Fubini argument).
\par
We have
$$
\hat{f}(t)=\eps^{m}e^{-it\cdot x_0} \hat{f}_0(\eps t) =\eps^m
e^{-it\cdot x_0} \prod_{1\leq j\leq m}{\hat{g}_0(\eps t_j)},\quad\quad
t\in \Rr^m.
$$
\par
Since 
$$
\hat{g}_0(t)=\frac{1}{\sqrt{2\pi}}
\frac{1}{(it)^j}\int_{\Rr}{g_0^{(j)}(x)e^{-itx}dx}
$$
for $t\not=0$ and $j\geq 0$ (by repeated integration by parts), we
find that
$$
|\hat{g}_0(t)|\ll \min(1,\Delta^j|t|^{-j}),
$$
the implied constant depending on $j$.
\par
Using the formula for $\hat{f}(t)$, selecting for given $t$ an index
$j$ so that $\|t\|_{\infty}=|t_j|$ and applying the second upper bound
above for this index only, if $|t_j|\geq 1$, we derive
\begin{equation}\label{eq-estimate-hatf}
  |\hat{f}(t)|\ll \min(\eps^m,\Delta^{B+m}\eps^{-B}
  \|t\|_{\infty}^{-B-m}),
  \quad\quad t\in\Rr^m,
\end{equation}
for any fixed $B\geq 1$, where the implied constant depends only on
$m$ and $B$.  In particular, for $\|t\|_{\infty}\leq 1$, we will use
simply the upper bound $|\hat{f}(t)|\leq
\eps^m\|\hat{f}_0\|_{\infty}\leq \eps^m$.
\par
We now proceed to approximate.  First of all, for any radius $R_N\geq
1$, we can estimate the contribution of those $t$ with $\|t\|>R_N$
using the estimate above with a value of $B\geq 1$ which will be
determined later.
After integrating over $\|t\|_{\infty}>R_N$, we obtain
$$
\int_{\|t\|> R_N}{\hat{f}(t)\Phi_N(t)dt} \ll \eps^{-B}\Delta^{m+B} R_N^{-B}
$$
for any $R_N\geq 1$. After selecting $R_N=\sigma_N^{1/B}\geq 1$, we
obtain
\begin{equation}\label{eq-largebis}
  \int_{\|t\|> R_N}{\hat{f}(t)\Phi_N(t)dt}\ll \eps^{-B}\Delta^{m+B}
  \sigma_N^{-1},
\end{equation}
for $N\geq 1$, the implied constant depending only on $f_0$.
\par
On the other hand, provided $B>1/a$, we
use~(\ref{eq-strong-convergence}) and~(\ref{eq-order-2}) to get
\begin{align*}
\int_{\|t\|\leq R_N}{\hat{f}(t)\Phi_N(t)dt}&=
\int_{\|t\|\leq R_N}{\hat{f}(t)\Phi(t)\exp(-Q_N(t)/2)
\Bigl\{
1+O\Bigl(\frac{1}{\gamma_N}\Bigr)
\Bigr\}dt}\\
&=\int_{\|t\|\leq R_N}{\hat{f}(t)\Phi(t)\exp(-Q_N(t)/2)dt}+
O\Bigl(\gamma_N^{-1}
\int_{\|t\|\leq R_N}{|\Phi(t)\hat{f}(t)|dt}
\Bigr)\\
&=\int_{\|t\|\leq R_N}{\hat{f}(t)\Phi(t)\exp(-Q_N(t)/2)dt}+
O(\eps^{m}\gamma_N^{-1}\exp(\beta R_N^A))
\end{align*}
for $N\geq 1$, using again the definition of $f$, the implied constant
depending on $\Phi$.
\par
By~(\ref{eq-uniformity}), the last term can be bounded by
$$
\eps^{m}\gamma_N^{-1}\exp(\beta R_N^A)\ll
\eps^{m}\exp(\beta\sigma_N^{A/B}-\alpha\sigma_N^{C}) \ll
\eps^{m}\sigma_N^{-1}
$$
for $N\geq 1$ if $B>A/C$, the implied constant depending on $(\alpha,
A,B, C)$.
\par
We then split the first term 
further in two parts, namely where $Q_N(t)\leq \kappa_N^2$, and where
$Q_N(t)>\kappa_N^2$. The parameter $\kappa_N$ will be chosen later, in
such a way that the region $\{\|t\|\leq 1\}$ (which is inside
$\{\|t\|\leq R_N\}$) contains the region $Q_N(t)\leq \kappa_N^2$
(which is a neighborhood of $0$ that contracts to $0$ as $N\ra
+\infty$, if $\kappa_N$ does not grow too fast, since it is an
ellipsoid with longest axis $\kappa_N/\sqrt{\delta_{m,N}}$).
\par
The second part of the integral is bounded by
\begin{align*}
  \int_{\|t\|\leq R_N,\ Q_N(t)>\kappa_N^2}{
    \hat{f}(t)\Phi(t)e^{-Q_N(t)/2}dt } &\ll
  \eps^{m}R_N^m\exp(\beta R_N^A-\kappa_N^2/2)
  \\&=\eps^{m}\exp\Bigl(\frac{m}{B}\log
  \sigma_N+\beta \sigma_N^{A/B}-\frac{\kappa_N^2}{2}\Bigr),
\end{align*}
and in the first part, we use the approximation
$$
\Phi(t)=1+O(\|t\|)
$$
for $\{\|t\|\leq 1\}$, coming from the $C^1$ assumption on $\Phi$, to
get
\begin{align}
  \int_{Q_N(t)\leq \kappa_N^2}{ \hat{f}(t)\Phi(t)e^{-Q_N(t)/2}dt }
  &=\int_{Q_N(t)\leq \kappa_N^2}{ \hat{f}(t)e^{-Q_N(t)/2}dt }
  +O\Bigl(\int_{Q_N(t)\leq \kappa_N^2}{ |\hat{f}(t)|\|t\|dt
  }\Bigr)\nonumber\\
  &=\int_{Q_N(t)\leq \kappa_N^2}{ \hat{f}(t)e^{-Q_N(t)/2}dt }
  +O\Bigl(\frac{\eps^{m}\kappa_N^{m+1}}{\sqrt{\delta_N\sigma_N}}\Bigr),
\label{eq-close-range}
\end{align}
the implied constant depending on $\Phi$, where the last integral was
estimated using
$$
\|t\|^2\leq \frac{1}{\delta_N}Q_N(t)\leq
\frac{\kappa_N^2}{\delta_N},\quad\quad
|t_i|\leq \frac{\kappa_N}{\sqrt{\delta_{i,N}}}.
$$
\par
We can rewind the computation for the first term, with the Gaussian
$G_N$ instead of $X_N$: for $N\geq 1$, we have
$$
\int_{Q_N(t)\leq \kappa_N^2}{ 
\hat{f}(t)e^{-Q_N(t)/2}dt }=\expect(f(G_N))
+O(\eps^{m}e^{-\kappa_N^2/2})
$$
where the implied constant is absolute, and then we write
\begin{align*}
  \expect(f(G_N)) &=
  \proba(\|G_N-x_0\|_{\infty}<\eps)+O(\proba(w\eps\leq
  \|G_N-x_0\|_{\infty}<\eps))\\
  &=\proba(\|G_N-x_0\|_{\infty}<\eps)+ O(\eps^m(1-w)\sigma_N^{-1/2})
\end{align*}
for $N\geq 1$, where the implied constant depends only on $m$ (the
last step is obtained using the density of the Gaussian $G_N$).
\par
Summarizing, we have found
\begin{multline*}
\proba(X_N\in U)\geq \proba(G_N\in U)+O\Bigl(\frac{\eps^m(1-w)}
{\sigma_N^{1/2}}
+\frac{\Delta^{m+B}}{\eps^m\sigma_N}\\
+\eps^m\exp\Bigl(\frac{m}{B}\log
\sigma_N+\beta\sigma_N^{A/B}-\frac{\kappa_N^2}{2}\Bigr)
+\frac{\eps^m\kappa_N^{m+1}}{\sqrt{\sigma_N\delta_N}}
\Bigr)
\end{multline*}
(where we recall that $\Delta^{-1}=1-w$).
\par
Now if $A/B<1/(m\mu)$, and
$$
\kappa_N=\sigma_N^{A/B},
$$
we have first (see~(\ref{eq-mmu})) the condition
$$
\{Q_N(t)\leq \kappa_N^2\}\subset \{\|t\|^2\leq
\frac{\kappa_N^2}{\delta_N}\}
\subset \{\|t\|\leq 1\},
$$
and moreover the third error term is absorbed in the second one, while
the last becomes
$$
\frac{\eps^{m}\kappa_N^{m+1}}{\sqrt{\sigma_N\delta_N}}\leq
\eps^{m}\sigma_N^{-\tfrac{1}{2}-\tfrac{1}{2m\mu}+\tfrac{(m+1)A}{B}}.
$$
\par
Thus if we select any $B>\max(a^{-1}, A/C, 3m(m+1)A\mu)$, the previous
conditions on $B$ hold, and we find that
\begin{equation}\label{eq-lower-bound-interm}
\proba(X_N\in U)\geq \proba(G_N\in U)+
O\Bigl(\frac{\eps^m(1-w)}{\sigma_N^{1/2}}
+\frac{\Delta^{m+B}}{\eps^m\sigma_N}
+\frac{\eps^m}{\sigma_N^{1/2+1/(6m\mu)}}
\Bigr).
\end{equation}
\par
Now, we attempt to select $w$ to equalize the error terms
involving it, i.e., so that
$$
\frac{\eps^m(1-w)}{\sqrt{\sigma_N}}=\frac{\Delta^{m+B}}{\eps^m \sigma_N},
$$
which translates to
$$
\Delta=(1-w)^{-1}=(\eps^{2m}\sigma_N^{1/2})^{1/(m+B+1)}
$$
and two cases arise:
\par
(i) If $\sigma_N^{1/2}>\eps^{-2m}$, we have $\Delta>1$ (as is
necessary to define $w$ in this way), and we obtain from the above
that 
$$
\proba(X_N\in U)\geq \proba(G_N\in U)+
O\Bigl(\frac{1}{\sigma_N^{1/2+1/D}}
\Bigr),
$$
where $D=2(m+1+B)$ (note that, since $A$ is assumed to be $\geq 1$, we
have $2(m+B+1)>6m\mu$.)
\par
(ii) If we have $\sigma_N^{1/2}\leq \eps^{-2m}$, we take simply
$w=1/2$ and obtain
$$
\frac{\eps^m(1-w)}{\sqrt{\sigma_N}}+\frac{\Delta^{m+B}}{\eps^m
  \sigma_N}\ll \frac{\eps^{-m}}{\sigma_N},
$$
where the implied constant depends on $m$ and $B$. 
\par
The combination of the two cases leads to the lower-bound
in~(\ref{eq-quantitative}). The upper bound is proved similarly, using
instead of $g_0$ a function which is $=1$ for $|x|\leq 1$ and $=0$ for
$|x|\geq 1+w$ for some suitable $w>0$; we leave the details to the
reader.
\end{proof}

\begin{remark}\label{rm-lower-bound}
  If we are interested in obtaining a lower bound only (which is the
  most interesting aspect in a number of applications), it is simpler
  and more efficient to fix, e.g., $w=1/2$, from the beginning of the
  proof. For
$$
U=\{x\in\Rr^m\,\mid\, \|x-x_0\|_{\infty}<\eps\},
$$
this leads for instance to
\begin{equation}\label{eq-quant-lower}
\proba(X_N\in U)\geq \proba(G_N\in U_-)+
O\Bigl(
\frac{1}{\sigma_N^{1/2+1/(2m\mu)-\gamma}}
\Bigr)
\end{equation}
for any $\gamma>0$ (by taking $B$ large enough depending on $\gamma$),
where
$$
U_-=\{x\in\Rr^m\,\mid\, \|x-x_0\|_{\infty}<\eps/2\}
$$
and the implied constant depends on $\Phi$, $(m, a,\alpha,C)$ and
$\gamma$.
\end{remark}

\begin{remark}\label{rm-tv}
  From the probabilistic point of view, this proposition gives one
  answer, quantitatively, to the following type of question: given a
  sequence $(X_N)$ of (real-valued) random variables and parameters
  $\sigma_N\ra +\infty$ such that $X_N/\sigma_N$ converges in law to a
  standard centered Gaussian variable (with variance $1$), to what
  extent is $X_N$ itself distributed like a Gaussian with variance
  $\sigma_N^2$?
\par
Here, we perform the comparison by looking at $\proba(X_N\in U)$, $U$
a fixed open set. And this shows clear limits to the Gaussian model:
for example, any integer-valued random variable $X_N$ will have
$\proba(X_N\in U)=0$ for any open set not intersecting $\Zz$, and yet
may satisfy a Central Limit Theorem (e.g., the $N$-th step of a
symmetric random walk on $\Zz$). Even more precisely, denoting
$$
d_K(X,Y)=\sup_{x\in\Rr}{|\proba(X\leq x)-\proba(Y\leq x)|}
$$
the Kolmogorov distance, there exist integer-valued random variables
$X_N$ with
\begin{equation}\label{eq-dist-k}
  d_K(X_N,G_N)\ll N^{-1/2},
\end{equation}
where $G_N$ is a centered Gaussian with variance $N$, which indicates
a close proximity -- and yet, again, $\proba(X_N\in U)=0$ if
$U\cap\Zz=\emptyset$. For instance, take $X_N$ with distribution
function $F_N(t)=\proba(X_N\leq t)$ given by
$$
F_N(t)=\proba(G_N\leq k)\quad\text{ for } k\leq t<k+1,\ k\in\Zz.
$$
\par
In particular, it is also impossible to prove results like
Theorem~\ref{th-2} using only this type of assumption on the
Kolmogorov distance. Of course, if we assume that
\begin{equation}\label{eq-impossible}
  d_K(X_N,G_N)\ll N^{-1/2}\phi(N)^{-1},
\end{equation}
with $\phi(N)\ra +\infty$ (arbitrarily slowly), we get
$$
\proba(X_N\in U)\geq \proba(G_N\in U)-2d_K(X_N,G_N)\gg N^{-1/2},
$$
for $U=[\alpha,\beta]$, $\alpha<\beta$. But such an assumption is
unrealistic in practice. For instance, if one assumes that $(X_N)$
converges in the mod-Gaussian sense with covariance $Q_N(t)=\sigma_N
t^2$, and if the limiting function $\Phi$ is $C^1$ and the convergence
holds in $C^1$ topology, one can straightforwardly estimate the
Kolmogorov distance\footnote{\ Using the Berry-Esseen inequality, see
  e.g.~\cite[\S 7.6]{tenenbaum}.}  by
$$
d_{K}(X_N,G_N)=\sup_{x\in\Rr}{|\proba(X_N\leq x)-\proba(G_N\leq x)|}
\ll \sigma_{N}^{-1/2},
$$
which is comparable to~(\ref{eq-dist-k}), but one can also check that
this can not be improved in general to something
like~(\ref{eq-impossible}). 
\par 
In Example~\ref{ex-eigenvalue-count} in Section~\ref{sec-examples}, we
will also describe a much deeper and more illuminating situation
concerning the limits of what can be hoped, even with something like
mod-Gaussian convergence.
\end{remark}

\begin{remark}\label{rm-variant}
  Other variants could easily be obtained. In particular, it is clear
  from the proof that if $\Phi$ has sub-gaussian growth, i.e., we can
  take $A=2$ in~(\ref{eq-order-2}), the results can be substantially
  improved. However, in our main applications, this condition
  fails. Also, one could use test functions $f$ which decay at
  infinity faster than polynomials to weaken the uniformity
  requirement in the convergence
  condition~(\ref{eq-strong-convergence}) (for instance, for $m=1$, it
  is possible to find $f$ which is smooth, non-negative and compactly
  supported in any fixed open interval and satisfies
$$
\hat{f}(t)\ll \exp(-|t|^{1-\eps})
$$
for any $\eps>0$, as constructed, e.g., in~\cite[Th.
1.3.5]{hormander} or~\cite{ingham}). Again, for our main unconditional
applications, our conditions hold with room to spare, so we avoided
this additional complexity. 

\end{remark}

\begin{remark}
  We can also introduce a linear term (corresponding roughly to the
  expectation of $X_N$) in addition to the covariance terms in the
  definition of mod-Gaussian convergence (as in~\cite{jkn} for $m=1$),
  but this amounts to saying that $(X_N)$ converges in the
  mod-Gaussian sense with covariance $(Q_N)$ and mean $(\xi_N)$,
  $\xi_N\in \Rr^m$, if the sequence $(X_N-\xi_N)$ converges in our
  original sense above. But note that the interpretation of a lower
  bound for $\proba(X_N-\xi_N\in U)$, as given by
  Theorem~\ref{th-2} for a fixed $U\subset \Rr^m$, is quite
  different, if $\xi_N$ is itself ``large''. Maybe one should see the
  statements in that case as giving natural examples of sets $U_N$ for
  which one knows that $\proba(X_N\in U_N)$ has the specific decay
  behavior $\sigma_N^{-1/2}$ as $N\ra +\infty$.
\par
In particular, lower bounds for $\proba(X_N-\xi_N\in U)$ do not give
control of $\proba(X_N\in U)$, and indeed this may be zero for all $N$
large enough (see, for instance the example in
Section~\ref{ssec-symplectic} below of values at $1$ of characteristic
polynomials of unitary symplectic matrices, which is always $\geq 0$).
\end{remark}

\section{Random unitary matrices and the zeta function}
\label{sec-examples}

We present now some applications of Theorem~\ref{th-2} (in particular,
proving Theorems~\ref{th-rmt} and~\ref{th-eulerproduct}).  We also
give an example that illustrates the limitations of such results,
suggesting strongly that one can not replace mod-Gaussian convergence
with the existence of the limits~(\ref{eq-converge}) only for $t$ in a
neighborhood of the origin.

\begin{example}\label{ex-keating-snaith}
  One of the canonical motivating examples of mod-Gaussian convergence
  is due to Keating and Snaith~\cite{keating-snaith}. Let
$$
X_N=\log \det(1-T_N),
$$
where $T_N$ is a Haar-distributed random unitary matrix in the compact
group $U(N)$; we view these random variables as $\Rr^2$-valued (via the
real and imaginary parts), so if $t=(t_1,t_2)$, we have
\begin{equation}\label{eq-inner-product}
t\cdot X_N=t_1\Reel(X_N)+t_2\Imag( X_N).
\end{equation}
\par
We first clarify the choice of the branch of logarithm: $X_N$ is
defined almost everywhere (when $1$ is not an eigenvalue of $T_N$),
and for $g\in U(N)$ with $\det(1-g)\not=0$, such that
$$
\det(1-Tg)=\prod_{1\leq j\leq N}{(1-\alpha_j T)},\quad\quad
|\alpha_j|=1,
$$
we define
$$
\log \det(1-g)=\lim_{\stacksum{r\ra 1}{r<1}}{\log \det(1-rg)}
=\sum_{1\leq j\leq N}\lim_{\stacksum{r\ra 1}{r<1}}{\log (1-r\alpha_j)},
$$
where the last logarithms are given by the Taylor expansion around
$0$. This is the same convention as in~\cite[par. after
(7)]{keating-snaith}. 
\par
Keating and Snaith show that $(X_N)$ satisfies
\begin{equation}\label{eq-ks}
\expect(e^{it\cdot X_N})=\prod_{1\leq j\leq N}{
\frac{\Gamma(j)\Gamma(j+it_1)}{\Gamma(j+\tfrac{1}{2}(it_1+t_2))
\Gamma(j+\tfrac{1}{2}(it_1-t_2))}
}
\end{equation}
for $t=(t_1,t_2)\in\Rr^2$ (note the asymmetry between $it_1$ and
$t_2$), see~\cite[eq. (71)]{keating-snaith}, taking into account a
slightly different normalization: their $t$ is our $it_1$ and their
$s$ is our $t_2$. It is also useful to observe that, in the connection
with T\"oplitz determinants (see~\cite{basor}), this characteristic
function is the $N$-th T\"oplitz determinant corresponding to the
symbol which is a pure Fisher-Hartwig singularity of type
$$
b(e^{i\theta})=
(2-2\cos\theta)^{it_1/2}\exp(i(\theta-\pi)t_2/2),\quad\quad
0<\theta<2\pi
$$
(this is denoted 
$$
t_{t_2/2}(e^{i\theta})u_{it_1/2}(e^{i\theta})=
\xi_{(it_1-t_2)/2}(e^{i\theta})\eta_{(it_1+t_2)/2}(e^{i\theta})
$$
in~\cite{e-s}, where the formula~(\ref{eq-ks}) is stated as Eq. (41);
see~\cite{bw} for two elementary computations of the corresponding
T\"oplitz determinants.)
\par
We rewrite this in terms of the Barnes function $G(z)$, as is
customary. We recall that $G$ is an entire function of order $2$, such
that $G(1)=1$ and $G(z+1)=\Gamma(z)G(z)$ for all $z$, and that its
zeros are located at the negative integers. In particular, it
satisfies
\begin{equation}\label{eq-ratio-gamma}
\prod_{j=1}^N{\Gamma(j+\theta)}=\frac{G(1+N+\theta)}{G(1+\theta)}
\end{equation}
for all $N\geq 1$ and $\theta\in\Cc$.
\par
Thus, we have
\begin{multline*}
\expect(e^{it\cdot X_N})=
\prod_{1\leq j\leq N}{
\frac{\Gamma(j)\Gamma(j+it_1)}{\Gamma(j+\tfrac{it_1+t_2}{2})
\Gamma(j+\tfrac{it_1-t_2}{2})}}=
\frac{G(1+\tfrac{it_1-t_2}{2})G(1+\tfrac{it_1+t_2}{2})}{G(1+it_1)}\\
\times
\frac{G(1+it_1+N)G(1+N)}{G(1+\tfrac{it_1-t_2}{2}+N)G(1+\tfrac{it_1+t_2}{2}+N)}
\end{multline*}
(see, e.g.,~\cite[eq.  (41)]{e-s}). We now get
from~\cite[Cor. 3.2]{e-s} that
\begin{align*}
\expect(e^{it\cdot X_N})&\sim N^{(it_1-t_2)(it_1+t_2)/4}
\frac{G(1+\tfrac{it_1-t_2}{2})G(1+\tfrac{it_1+t_2}{2})}{G(1+it_1)}
\\
&=\exp(-Q_N(t)/2)
\frac{G(1+\tfrac{it_1-t_2}{2})G(1+\tfrac{it_1+t_2}{2})}{G(1+it_1)},
\end{align*}
where 
$$
Q_N(t_1,t_2)=\delta_N(t_1^2+t_2^2),\quad\quad
\delta_N=\tfrac{1}{2}\log N,
$$ 
hence we have complex mod-Gaussian convergence with limiting function
\begin{equation}\label{eq-phig}
\Phi_g(t_1,t_2)=\frac{G(1+\tfrac{it_1-t_2}{2})
G(1+\tfrac{it_1+t_2}{2})}{G(1+it_1)}.
\end{equation}
\par
Here we can take $\mu=1$ in~(\ref{eq-grow-not-too-much}).

\begin{remark} Note that this is \emph{not} the product of the two
  individual limiting functions for mod-Gaussian convergence of the
  real and imaginary parts of $X_N$ separately (which are
  $\Phi_g(t_1,0)$ and $\Phi_g(0,t_2)$), although after normalizing,
  one obtains convergence in law of
$$
(\Reel(X_N)/\sqrt{\delta_N},\Imag(X_N)/\sqrt{\delta_N})
$$
to \emph{independent} standard Gaussian variables, as noted by Keating
and Snaith.
\end{remark}

We now check that Theorem~\ref{th-2} can be applied to the sequence of
random variables $(X_N)$. The fact that $\Phi_g$ is of class $C^1$ on
$\Rr^2$ is clear in view of the analytic properties of the Barnes
function. Condition (3) is also obvious.  A uniformity estimate
like~(\ref{eq-strong-convergence}) is not found
in~\cite{keating-snaith} or~\cite{e-s}, though it is proved for $t$ in
a fixed compact region of $\Rr^2$ in~\cite[Cor. 3.2]{e-s}. In
Proposition~\ref{pr-app} in the Appendix, we prove
$$
\Phi_N(t)=\Phi(t)e^{-Q_N(t)/2}\Bigl(1+
O\Bigl(\frac{1+\|t\|^3}{N}\Bigr)\Bigr)
$$
for $\|t\|\leq N^{1/6}$.  In view of $\sigma_N=(\log N)^2/4$, this is
compatible with~(\ref{eq-strong-convergence})
and~(\ref{eq-uniformity}), with $a$ arbitrarily large, $C=\demi$ and
$A$ (defined by~(\ref{eq-order-2})) can be taken to be any $A>2$. Thus
the constant $D$ in~(\ref{eq-quantitative}) can be any number
$$
D>2(3+\max(4,36))=78.
$$
\par
Moreover, since
$$
\tilde{Q}_N(x_1,x_2)=\frac{x_1^2+x_2^2}{\demi \log N},
$$
we see by using Remark~\ref{rm-lower-bound}
(namely,~(\ref{eq-quant-lower})) and~(\ref{eq-lower-easy}) that we
have the following corollary:

\begin{corollary}
  For $0<\eps<1$ we have
\begin{equation}\label{eq-dense-quant-log}
\proba(|X_N-z_0|<\eps)\gg \eps^2(\log N)^{-1}
\end{equation}
for all $N$ with
$$
N\gg \max\bigl\{\exp(|z_0|^2),\exp(C\eps^{-9})\bigr\}
$$
where both implied constants are absolute.
\end{corollary}

(The first condition on $N$ ensures the main term is $\gg \eps^2(\log
N)^{-1}$, while the second ensures that the error term is smaller; we
have taken $2/9=1/(2m\mu)-\gamma=1/4-1/36$ for definiteness; any
number $>8$ can replace $9$).
\par
Now we appeal to the following elementary lemma:

\begin{lemma}\label{lm-log}
Let $z_0\in\Cc^{\times}$ and $\eps>0$, and denote
$$
w_0=\log |z_0|+i\theta_0,\quad\quad \theta_0=\Arg(z_0)\in ]-\pi,\pi].
$$
\par
Then, provided $\eps\leq |z_0|$, we have
$$
|e^w-z_0|<\eps
$$
for all $w\in \Cc$ such that
$$
|w-w_0|<\frac{\eps}{2|z_0|}.
$$
\end{lemma}

For given $z_0\in \Cc$, non-zero, we get from this
and~(\ref{eq-dense-quant-log}), applied to $\log |z_0|+i\Arg(z_0)$ and
to $\eps/|z_0|$ instead of $\eps$, the following explicit form of
Theorem~\ref{th-rmt}:

\begin{theorem}\label{th-rmt-quant}
  Let $z_0\in \Cc^{\times}$ be arbitrary, $\eps>0$ such that $\eps\leq
  |z_0|$. We have
$$
  \proba(|\det(1-T_N)-z_0|<\eps)\gg 
\Bigl(\frac{\eps}{|z_0|}\Bigr)^2
\frac{1}{\log N}
$$
for $N\geq N_0(z_0,\eps)$, where
\begin{gather*}
  N_0(z_0,\eps)\ll \max\Bigl\{\exp\bigl((\log |z_0|)^2\bigr),
  \exp\Bigl(C\Bigl(\frac{\eps}{2|z_0|}\Bigr)^{-9}\Bigr)\Bigr\}
\end{gather*}
where $C$ and the implied constants are absolute.
\end{theorem}


\begin{remark}
  It follows from asymptotic formulas for the Barnes function,
  e.g.~\cite{fl}, that 
  we have 
$$
\tfrac{1}{\|t\|^2}\log |\Phi_g(t)|\asymp \log (2\|t\|),
$$ 
which is illustrated in Figure~1. This super-gaussian behavior is the
main cause of difficulty in the proof of Theorem~\ref{th-2}.
\begin{figure}[ht]
\centering
\includegraphics[width=4in]{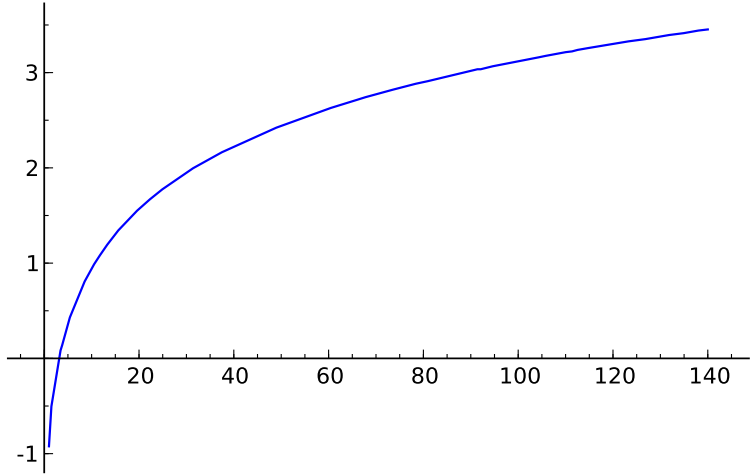}
\caption{Graph of $\tfrac{1}{t_2^2}\log |\Phi_{g}(1,t_2)|$, $1\leq
  t_2\leq 140$}
\end{figure}
\end{remark}


\begin{remark}
Note that if one only wants to say that $\det(1-T_N)$, for $N$
growing, has dense image in $\Cc$, much simpler topological arguments
suffice.
\end{remark}

\end{example}

\begin{example}
  Another conspicuous example of mod-Gaussian convergence is the
  arithmetic Euler factor in the moment conjecture for
  $\zeta(1/2+it)$.  In~\cite[\S 4.1]{jkn}, we considered this factor
  for the real part $\log|\zeta(\demi+it)|$ only, and we first
  generalize this as in the previous section.  
\par
Consider a sequence $(X_p)$ of independent random variables uniformly
distributed on the unit circle and indexed by prime numbers, and let
$$
L_N=-\sum_{p\leq N}{\log(1-p^{-1/2}X_p)}
$$
where the logarithm is given here by the Taylor expansion around
$0$. For each individual term $E_p=-\log(1-p^{-1/2}X_p)$, we have
$$
\expect(e^{it\cdot E_p})=
\frac{1}{2\pi}\int_0^{2\pi}{
(1-ae^{i\theta})^{-\tfrac{1}{2}(t_2+it_1)}
(1-ae^{-i\theta})^{\tfrac{1}{2}(t_2-it_1)}
d\theta}
$$
with $a=p^{-1/2}$. Expanding by the binomial theorem and picking up
the constant term in the expansion in Fourier series, we obtain
\begin{align*}
\expect(e^{it\cdot E_p})&=
\sum_{j\geq 0}{a^{2j}
\binom{-\tfrac{1}{2}(t_2+it_1)}{j}
\binom{\tfrac{1}{2}(t_2-it_1)}{j}
}\\
&=
\sum_{j\geq 0}{
\frac{(\tfrac{1}{2}(t_2+it_1))_j(\tfrac{1}{2}(it_1-t_2))_j}
{(j!)^2}a^{2j}
}\\
&=\hyperg{\tfrac{1}{2}(it_1+t_2)}{\tfrac{1}{2}(it_1-t_2)}{1}{a^2}
\end{align*}
in terms of the Gauss hypergeometric function.
\par
Arguing as in~\cite{jkn}, we see now that
\begin{align*}
\expect(e^{it\cdot L_N})&=\prod_{p\leq N}{
\hyperg{\tfrac{1}{2}(it_1+t_2)}{\tfrac{1}{2}(it_1-t_2)}{1}{p^{-1}}
}\\
&=\prod_{p\leq
  N}{\Bigl(1-\frac{t_1^2+t_2^2}{4}\frac{1}{p}+O\Bigl(\frac{1}{p^{2}}
\Bigr)\Bigr)},
\end{align*}
and hence, denoting
$$
\delta_N=-\frac{1}{2}\sum_{p\leq N}{\log (1-p^{-1})}\sim 
\frac{1}{2}\log\log N,\quad\quad
Q_N(t)=\delta_N (t_1^2+t_2^2)=\delta_N\|t\|^2,
$$
we get
$$
\expect(e^{it\cdot L_N})\sim \exp(-Q_N(t)/2)\Phi_a(t),
$$
as $N\ra +\infty$, with limiting function
\begin{equation}\label{eq-euler-factor}
\Phi_a(t)=\prod_{p}{\Bigl(1-\frac{1}{p}\Bigr)^{-\|t\|^2/4}
\hyperg{\tfrac{1}{2}(it_1+t_2)}{\tfrac{1}{2}(it_1-t_2)}{1}{p^{-1}}
}.
\end{equation}
\par
We have here also $\mu=1$ in~(\ref{eq-grow-not-too-much}). Now, to
check the uniformity required in~(\ref{eq-strong-convergence}), we
write
$$
\expect(e^{it\cdot L_N})=\exp(-Q_N(t)/2)\Phi_a(t)R_N(t)
$$
with
$$
R_N(t)=\prod_{p>N}{\Bigl(1-\frac{1}{p}\Bigr)^{-\|t\|^2/4}
  \hyperg{\tfrac{1}{2}(it_1+t_2)}{\tfrac{1}{2}(it_1-t_2)}{1}{p^{-1}}}.
$$
\par
If we expand the $p$-factor using the binomial theorem, we obtain
$$
1+\sum_{j\geq 2}{\frac{1}{p^j}
\sum_{a+b=j}{\frac{(\tfrac{1}{2}(it_1+t_2))_a(\tfrac{1}{2}(it_1-t_2))_a}
{(a!)^2}
\binom{-\|t\|^2/4}{b}}},
$$
and if we assume that $\|t\|\leq A$ with $A\geq 1$, crude bounds show
that this $p$-factor is
$$
1+O\Bigl(\sum_{j\geq 2}{
\frac{jA^{2j}}{p^j}
}\Bigr),
$$
where the implied constant is absolute, so that if $\|t\|\leq
N^{1/8}$, for instance, we get
$$
R_N(t)=\prod_{p>N}{\Bigl(1+O\Bigl(\sum_{j\geq
    2}{jp^{-3j/4}}\Bigr)\Bigr)}
=1+O(N^{-1/2}).
$$
\par
Although this is crude, it already gives much more
than~(\ref{eq-strong-convergence}), both in terms of range of
uniformity and sharpness of approximation.
\par
Since Condition (3) is also obviously valid here, Theorem~\ref{th-2}
(or rather~(\ref{eq-quant-lower})) applies with $A$ any real number
$>2$, $a$ and $C$ arbitrarily large, and shows that
$$
\proba(|L_N-z_0|<\eps)\gg \eps^2(\log\log N)^{-1}
$$
for any $z_0\in\Cc$ and $\eps<1$, provided
$$
N\gg \max\Bigl(\exp(\exp(|z_0|^2)),\exp(\exp(C\eps^{-9}))\Bigr)
$$
for some large constant $C\geq 1$.
\par
From this we deduce easily the more arithmetic-looking statement of
Theorem~\ref{th-eulerproduct}.  Indeed, let $P_N(t)$ be given
by~(\ref{eq-pnt}). For fixed $N$, it is well-known that the random
variables $t\mapsto P_{N}(t)$ on the probability spaces
$([0,T],T^{-1}\lambda)$ converge in law, as $T\ra +\infty$, to
$$
\tilde{P}_N=\prod_{p\leq N}{(1-p^{-1/2}X_p)^{-1}}=\exp(L_N),
$$
where $X_p$ are as above (independent and uniformly distributed on the
unit circle; the independence is due to the fundamental theorem of
arithmetic). For any open set $V$,\footnote{\ Because we do not know
  if the probability that $\tilde{P}_N$ is in the boundary of $V$ is
  zero or not, we do not claim -- or need to claim -- an equality;
  see, e.g.,~\cite[Th.  2.1, (iv)]{billingsley}.} it follows that
$$
\liminf_{T\ra +\infty}{\frac{1}{T}\lambda(\{t\leq T\,\mid\, P_N(t)\in
  V\})} \geq \proba(\tilde{P}_N\in V).
$$
\par
Applying Lemma~\ref{lm-log} as in the proof of Theorem~\ref{th-rmt-quant},
we obtain Theorem~\ref{th-eulerproduct} with
\begin{gather*}
N_0(z_0,\eps)\ll\max\Bigl\{\exp\Bigl(\exp\Bigl((\log |z_0|)^2\Bigr)\Bigr),
\exp\Bigl(\exp\Bigl(C\Bigl(\frac{\eps}{2|z_0|}\Bigr)^{-9}\Bigr)\Bigr)
\Bigr\}
\end{gather*}
for some absolute constant $C$.

\begin{remark}
  Again, the density of values of $P_N(t)$ for $N\geq 1$ and $t\in\Rr$
  (or of $\tilde{P}_N$, which amounts to the same thing) is an easier
  matter that can be dealt with using topological tools.
\end{remark}

\end{example}

\begin{example}\label{ex-zeta}
  The two previous examples are of course motivated by their
  conjectural relation with the behavior of the Riemann zeta function
  on the critical line (this is the arithmetic essence
  of~\cite{keating-snaith}). Indeed, Keating and Snaith conjecture
  that:

\begin{conjecture}\label{cj-ks}
Define $\log \zeta(1/2+iu)$, when $u\in\Rr$ is not the ordinate of a
non-trivial zero of $\zeta(s)$, by continuation along the horizontal
line $\Imag(s)=u$, with limit $0$ when $\Reel(s)\ra +\infty$.
\par
For any $t=(t_1,t_2)\in\Rr^2$, we have
$$
\frac{1}{T}\int_0^T{e^{it\cdot \log\zeta(1/2+iu)}du}=
\Phi_a(t)\Phi_g(t)\exp\Bigl(-\frac{t^2}{2}(\log\log T)\Bigr)(1+o(1))
$$
as $T\ra +\infty$, where $\cdot$ is the inner product on $\Rr^2$ as
in~\emph{(\ref{eq-inner-product})}.
\end{conjecture}

Hence, we see in particular that the following holds:

\begin{corollary}\label{cor-zeta}
  Assume there exist $\alpha>0$, $\delta>0$ and $\theta>0$ such that
  Conjecture~\emph{\ref{cj-ks}} holds with the error term $o(1)$
  replaced by
$$
\exp(-\alpha (\log\log T)^{\delta})
$$
uniformly for $\|t\|\leq (\log\log 6T)^{\theta}$. Then the set of
values $\zeta(1/2+it)$ is dense in the complex plane. In fact, there
exists $C>0$, $D\geq 0$, such that, for any $z_0\in \Cc^{\times}$ and
$\eps\leq |z_0|$, there exists $t$ with
$$
0\leq t\ll \max\Bigl\{\exp\bigl(\exp\bigl((\log |z_0|)^2\bigr)\bigr),
\exp\bigl(\exp\Bigl(C\Bigl(\frac{\eps}{2|z_0|}\Bigr)^{-D}
\Bigr)\bigr)\Bigr\},
$$
such that
$$
|\zeta(\demi + it)-z_0|<\eps.
$$
\end{corollary}

Of course, such a strong conjecture concerning the imaginary moments
of $\zeta(1/2+it)$ looks quite hopeless at the current time: there is
no known non-trivial result available, even assuming the Riemann
Hypothesis. But Example~\ref{ex-eigenvalue-count} below suggests that
(with this approach) it is indeed necessary to require that the
characteristic function converge uniformly for $t$ in a region growing
with $T$. In~\cite{dkn}, jointly with F. Delbaen, we will explain how
the weaker qualitative statement
$$
\frac{1}{T}\lambda(\{u\in [0,T]\,\mid\, \zeta(\demi+iu)\in V\})\gg 
\frac{1}{\log\log T}
$$
for a fixed open set $V$ and for $T$ large enough (which of course
suffices to give a positive answer to Ramachandra's question) can be
proved under much weaker assumptions than a uniform version of
Conjecture~\ref{cj-ks}.
\par
\begin{rem}
  Another remark concerning Conjecture~\ref{cj-ks} has to do with the
  factored form of the limiting function $\Phi_a(t)\Phi_g(t)$, which
  seems to imply some asymptotic independence property. Recall that
  the real and imaginary parts of $\Phi_a$ and $\Phi_g$ are themselves
  asymptotically independent \emph{after renormalization}, but are not
  products of the limiting functions for the two parts separately. So
  Conjecture~\ref{cj-ks}, if correct, is evidence of quite particular
  probabilistic behavior.\footnote{\ The fourth moment of
    $\zeta(1/2+it)$ and a few other results do provide evidence of a
    factored limiting function, with ``random matrix'' term split from
    the Euler factor. The mod-Poisson analogy is also consistent with
    this, in the case of the number of prime factors of an integer, as
    discussed in detail in~\cite[\S 4, 5, 6]{mod-poisson}.}
\end{rem}

\end{example}

\begin{example}\label{ex-eigenvalue-count}
  Our assumptions in Theorem~\ref{th-2} are probably not optimal. We
  now describe an illuminating (counter)-example in the direction of
  understanding when a result like this could be true.
  \par
We again look at random matrices $T_N$ in the compact group $U(N)$ (as
in Example~\ref{ex-keating-snaith}), but this time we consider the
random variables counting the number of eigenvalues in certain fixed
arcs of the unit circle: fix $\gamma\in ]0,1/2[$, and let
$$
I=\{e^{2i\pi\theta}\,\mid\, |\theta|\leq \gamma\}\subset \Cc.
$$
\par
Then let $X_N$ be the number of eigenvalues $\vartheta$ of $T_N$ such
that $\vartheta\in I$. Note that $X_N$ is an integer-valued random
variable. It was proved by Costin and Lebowitz that 
$$
\frac{X_N-2\gamma N}{\pi^{-1}\sqrt{\log N}}
$$
converges in law to a standard normal random variable.
Wieand~\cite{wieand} gave a proof\footnote{\ Including a more general
  result concerning the joint distribution of the number of
  eigenvalues in more than one interval.} based on asymptotics of
T\"oplitz determinants with discontinuous symbols; as noted by Basor,
this gives the asymptotic
$$
\expect(e^{it(X_N-2\gamma N)})\sim 
\exp\Bigl(-\frac{t^2}{2}\frac{1}{\pi^2}\log N\Bigr)
(2-2\cos
4\pi\gamma)^{\tfrac{t^2}{4\pi^2}}G\Bigl(1-\frac{t}{2\pi}\Bigr)
G\Bigl(1+\frac{t}{2\pi}\Bigr)
$$
as $N\ra +\infty$, for all $t$ \emph{with} $|t|<\pi$ (see,
e.g.,~\cite[Th. 5.47]{e-s}, applied with $N=2$, $\alpha_1=\alpha_2=0$,
$\beta_1=\tfrac{t}{2\pi}$, $\beta_2=-\tfrac{t}{2\pi}$, and the
condition on $t$ is equation (5.79) in loc. cit., or~\cite[p.
331]{basor}). This asymptotic is of course of the
form~(\ref{eq-converge}) for these values of $t$, but the restriction
$|t|<\pi$ is necessary, since the characteristic function of $X_N$ is
$2\pi$-periodic for all $N$. The convergence is sufficiently uniform
for $t$ close to $0$ to allow the deduction of the renormalized normal
behavior (as Wieand did, using the Laplace transform instead of the
characteristic function), but when $\gamma$ is rational, the set of
possible values of $X_N-2\gamma N$ for $N\geq 1$ is a discrete set in
$\Rr$.\footnote{\ For what it's worth, one may mention that the
  density of values $X_N-2\gamma N$ is true for irrational $\gamma$,
  by Dirichlet's approximation theorem, and by the existence of
  matrices in $U(N)$ where the number of eigenvalues in $I$ takes any
  value between $0$ and $N$.}
\end{example}

\section{Distribution of central values of $L$-functions over finite
  fields}
\label{sec-ff}

We now consider examples related to $L$-functions over finite
fields. Our main input will be deep results of Deligne and Katz, and
we are of course motivated by the philosophy of Katz and
Sarnak~\cite{katzsarnak}.
\par
The goal is to make statements about the distribution of values at the
central point $1/2$ of $L$-functions over finite fields. The appealing
aspect is that these form discrete sets, hence proving that they are
dense in $\Cc$ (as in Theorem~\ref{th-ff}), for instance, is obviously
interesting and meaningful.  We consider examples of our results for
the three basic symmetry types in turn: unitary, symplectic, and
orthogonal. For the last two, this means first obtaining a suitable
analogue of Example~\ref{ex-keating-snaith}. The corresponding
limiting functions have already been studied in some respect by
Keating-Snaith~\cite{ks-l-functions} and
Conrey-Farmer~\cite{conrey-farmer}, though our expressions seem
somewhat more natural.

\subsection{Unitary symmetry}
\label{ssec-unitary}

Let $\Fp_q$ be a finite field with $q$ elements. Unitary symmetry
arises (among other cases) for certain types of one-variable
exponential sums over finite fields, which are associated to Dirichlet
characters of $\Fp_q[X]$, which we now describe; these will lead to a
proof of Theorem~\ref{th-ff}.
\par
Let $\Fp_q$ be a finite field with $q$ elements of characteristic
$p\not=0$. A \emph{Dirichlet character} modulo $g\in\Fp_q[X]$ is a map
$$
\eta\,:\, \Fp_q[X]\ra \Cc,
$$
defined by
$$
\eta(f)=
\begin{cases}
0&\text{ if   $f$ and $g$ are not coprime}\\
\underline{\eta}(f)&\text{ otherwise},
\end{cases}
$$
where $\underline{\eta}$ is a group homomorphism
$$
\underline{\eta}\,:\, (\Fp_q[X]/g\Fp_q[X])^{\times}\lra \Cc^{\times}.
$$
\par
This character is non-trivial if $\underline{\eta}\not=1$, and primitive if it
can not be defined (in the obvious way) modulo a proper divisor of
$g$. The associated $L$-function is defined by the Euler product
$$
L(s,\eta)=\prod_{\pi}{(1-\eta(\pi)|\pi|^{-s})^{-1}},
$$
for $s\in\Cc$, where the product ranges over monic irreducible
polynomials in $\Fp_q[X]$ and $|\pi|=q^{\deg(\pi)}$. One shows quite
easily that this is in fact a polynomial (which we denote $Z(\eta,T)$)
in the variable $T=q^{-s}$ of degree $\deg(g)-1$, if $\eta$ is
primitive modulo $g$ and non-trivial.
\par
The examples used in proving Theorem~\ref{th-ff} arise from the
following well-known construction. For any integer $d\geq 1$, with
$p\nmid d$, any non-trivial multiplicative character
$$
\chi\,:\,\Fp_q\ra \Cc^{\times}
$$ 
such that $\chi^d\not=1$, and any squarefree polynomial $g\in
\Fp_q[X]$ of degree $d$, we let
$$
S(\chi,g)=\sum_{x\in\Fp_q}{\chi(g(x))},
$$
where $\chi(0)$ is defined to be $0$. These are multiplicative
exponential sums, and have been studied intensively, due in part to
their many applications to analytic number theory (for their
generalizations to multiple variables, see the
paper~\cite{katz-nonsing} of Katz).
\par
It is also well-known that one can construct a non-trivial Dirichlet
character $\eta=\eta(g,\chi)$, primitive modulo $g$, such that
$$
Z(\eta,T)=\exp\Bigl(\sum_{m\geq 1}{\frac{S_m(\chi,g)}{m}T^m}\Bigr),
$$
where $S_m(\chi,g)$ denotes the ``companion'' sums over extensions of
$\Fp_q$, namely
$$
S_m(\chi,g)=
\sum_{x\in \Fp_{q^m}}{
\chi(N_{\Fp_{q^m}/\Fp_q}(g(x)))},
$$
where $N_{\Fp_{q^m}/\Fp_q}$ is the norm map. We will denote
$L(s,g,\chi)$ the corresponding $L$-function.
\par
Moreover, we have the Riemann Hypothesis for these $L$-functions (due
to A. Weil), which gives the link with random unitary matrices: there
exists a unique conjugacy class $\theta_{\chi,g}(q)$ in the unitary
group $U(d-1)$ such that
$$
L(s+\demi,g,\chi)=\det(1-q^{-s}\theta_{\chi,g}(q)),
$$
(so that, in particular, we recover the Weil bound
$$
|S(\chi,g)|\leq (d-1)q^{1/2},
$$
by looking at the trace of $\theta_{\chi,g}(q)$).  For all this, one
can see, for instance,~\cite[\S 4.2]{kow-expsums}, which contains a
self-contained account.
\par
We will first prove the following theorem, which is clearly a stronger
form of Theorem~\ref{th-ff} in view of the preceding remarks:

\begin{theorem}\label{th-ff-strong}
  For $d\geq 1$ and $t\in\Zz$, let $g_{d,t}=X^d-dX-t\in\Zz[X]$. For
  $p$ prime, let $X(p)$ denote the set of pairs $(\chi,t)$ where
  $\chi$ is non-trivial character of $\Fp_p$ and $t\in\Fp_p$.
\par
Let $z_0\in\Cc^{\times}$ and $\eps>0$ with $\eps\leq |z_0|$ be
given. For all integers $d\geq d_0(z_0,\eps)$, we have
$$
\liminf_{p\ra +\infty}\frac{|\{(\chi,t)\in X(p)\,\mid\, \chi^d\not=1,
  |L(\demi,g_{d,t},\chi)-z_0|<\eps\}|}{|X(p)|} \gg
\Bigl(\frac{\eps}{|z_0|}\Bigr)^2 \frac{1}{\log d},
$$
where
\begin{gather*}
d_0(z_0,\eps)\ll \max\Bigl\{ \exp\bigl((\log |z_0|)^2\bigr),
\exp\Bigl(C\Bigl(\frac{\eps}{|z_0|}\Bigr)^{-9}\Bigr)\Bigr\},
\end{gather*}
where $C\geq 0$ and the implied constants are absolute.
\end{theorem}

This result depends on the mod-Gaussian convergence for characteristic
polynomials on $U(N)$ (i.e., on
Example~\ref{ex-keating-snaith}). Indeed, denoting by $U(N)^{\sharp}$
the space of conjugacy classes in $U(N)$, we have the following:

\begin{theorem}\label{th-katz}
  For any integer $d>5$, any odd prime $p$ with $p\nmid d(d-1)$, the
  conjugacy classes
$$
\{\theta_{\chi,g_{d,t}}(p)\,\mid\, \chi\mods{p},\ \chi\not=1,\text{ and }
t\in\Fp_p\text{ with } t^{d-1}-(1-d)^{d-1}\not=0\mods{p}\}
$$ 
become equidistributed in $U(d-1)^{\sharp}$ as $p\ra +\infty$, with
respect to Haar measure.
\end{theorem}

\begin{proof}\
  This is an easy consequence of results of Katz
  (see~\cite[Th. 5.13]{katz-act}), the only ``twist'' being the extra
  averaging over all non-trivial Dirichlet characters to obtain
  unitary instead of special unitary (or similar) equidistribution.
\par
First of all, it is easy to check that if $p\nmid d(d-1)$ and
$t\in\Fp_p$ is such that $t^{d-1}\not=(1-d)^{d-1}$, the polynomial
$g_{d,t}=X^d-dX-t\in \Fp_p[X]$ is a strong Deligne polynomial in one
variable (in the language of~\cite{katz-act}, these are called
``weakly-supermorse'' polynomials). Hence the conjugacy classes in the
statement are well-defined.
\par
For simplicity, denote $\mathcal{U}$ the open subset of the affine
$t$-line where $t^{d-1}\not=(1-d)^{d-1}$.
\par
Now, according to the Weyl equidistribution criterion, we must show
that
$$
\lim_{p\ra +\infty}\frac{1}{p-2}\sums_{\chi\mods{p}}{
\frac{1}{|\mathcal{U}(\Fp_{p})|}
\sum_{t\in\mathcal{U}(\Fp_{p})}{
\Tr\Lambda(\theta_{\chi,g_{d,t}}(p))
}}=0.
$$
for any (fixed) non-trivial irreducible unitary representation
$\Lambda$ of the compact group $U(d-1)$.
\par
We isolate in the sum those characters $\chi$ where $\chi^{2d}=1$:
there are at most $2d$ of them. For any other character
$\chi\mods{p}$, the inner sum over $t\in \mathcal{U}(\Fp_p)$ is of the
type handled by the Deligne Equidistribution Theorem. Let $k=k(\chi)$
be the order of the Dirichlet character $\chi\chi_2$, where $\chi_2$
is the real character modulo $p$. By~\cite[Th.  5.13, (2)]{katz-act}
(the restriction $\chi^{2d}\not=1$ ensures the assumptions hold),
provided $p\nmid d(d-1)$, e.g., $p>d(d-1)$, the relevant geometric
monodromy group is equal to
$$
GL_{k}(d-1)=
\{g\in GL(d-1)\,\mid\, \det(g)^{k}=1\},
$$
with maximal compact subgroup
$$
U_{k}(d-1)=\{g\in U(d-1)\,\mid\, \det(g)^{k}=1\}.
$$
\par
For simplicity, we write $U=U(d-1)$, $U_{k}=U_{k}(d-1)$.  The
restriction of $\Lambda$ to $U_{k}$ is a finite sum of irreducible
representations of this group (possibly including trivial
components). 
Applying~\cite[Th.  9.2.6, (5)]{katzsarnak} to each of the non-trivial
ones (and the obvious identity for the trivial components), we find
that
$$
\frac{1}{|\mathcal{U}(\Fp_{p})|}\sum_{t\in\mathcal{U}(\Fp_{p})}{
\Tr\Lambda(\theta_{\chi,g_{d,t}}(p))
}=\langle \Lambda\mid U_{k},1\rangle+
O((\dim\Lambda)dp^{-1/2})
$$
where the implied constant is \emph{absolute} (this is because we have
a one-parameter family, so we can apply~\cite[9.2.5]{katzsarnak} and
the fact proved in~\cite[5.12]{katz-act} that the relevant sheaf is
everywhere tame, so the Swan-conductor contribution is zero; the
parameter curve $\mathcal{U}$ has $d$ points at infinity, which give
the factor $d$ above).  
\par
Using Frobenius reciprocity or direct integration (using,
e.g,~\cite[Lemma AD.7.1]{katzsarnak}), we find that the multiplicity
of the trivial representation in $\Lambda$ (restricted to $U_{k}$)
satisfies
$$
\langle \Lambda\mid U_{k},1\rangle
=\sum_{h\in\Zz}{\langle \Lambda,\det(\cdot)^{hk}\rangle}
=\begin{cases}
1&\text{ if } \Lambda=\det(\cdot)^{hk}\text{ for some } h\in\Zz-\{0\},\\
0&\text{ otherwise}.
\end{cases}
$$
\par
For a given $\Lambda$, this is equal to $1$ only if $\dim\Lambda=1$,
so $\Lambda=\det(\cdot)^r$ for some $r\in\Zz-\{0\}$, and if $\chi$ is
such that $k(\chi)\mid r$.  The number of such characters $\chi$ is
therefore $\ll 1$, the implied constant depending on $\Lambda$. Hence
we find, after adding back the characters with $\chi^{2d}=1$, that
$$
\frac{1}{p-2}\sums_{\chi\mods{p}}{
\frac{1}{|\mathcal{U}(\Fp_{p})|}
\sum_{t\in\mathcal{U}(\Fp_{p})}{
\Tr\Lambda(\theta_{\chi,g_{d,t}}(p))
}}\ll \frac{d}{p}+\frac{(\dim \Lambda)d}{p^{1/2}},
$$
where the implied constant depends on $\Lambda$. This confirms the
claimed equidistribution.
\end{proof}

\begin{proof}[Proof of Theorem~\ref{th-ff-strong}]
  This an easy consequence of Theorem~\ref{th-katz}: for any open set
  $V\subset \Cc$, we have first that
$$
|\mathcal{U}(\Fp_p)|\sim p
$$
as $p$ goes to infinity, and then we can write
\begin{multline*}
  \liminf_{p\ra +\infty}{
    \frac{1}{|X(p)|}|\{(\chi,t)\in X(p)\,\mid\,
\chi^d\not=1,\quad L(\demi,g_{d,t},\chi)\in
      V \}|}\\
  \geq \mu_{d-1}(\{g\in U(d-1)\,\mid\, \det(1-g)\in V\}),
\end{multline*}
and then we apply Theorem~\ref{th-rmt-quant}. 
\end{proof}

\begin{remark}
  In Theorem~\ref{th-katz}, we performed the average over $\chi$,
  because for ``standard'' families of exponential sums (those
  parametrized by points of algebraic varieties), the (connected
  component of the) geometric monodromy group is always semisimple, so
  its center is finite and its maximal compact subgroup can never be
  $U(N)$. However, one may expect that $t$ could be fixed in the
  example above, and that (for instance) the conjugacy classes
$$
\{\theta_{\chi}(p)=\theta_{\chi,X^d-dX-1}\,\mid\, \chi\mods{p}\text{
  non-trivial}\}
$$
corresponding to the exponential sums
$$
S(\chi)=\sum_{x\in\Fp_p}{\chi(x^d-dx-1)}
$$
(where $d>6$, $p\nmid d(d-1)$) would be already equidistributed in
$U(d-1)^{\sharp}$ as $p\ra +\infty$. 
\par
Very recently, N. Katz~\cite{katz-mellin} has indeed shown that for
such families there is always an a-priori ``Sato-Tate law'', i.e.,
that the conjugacy classes become equidistributed in $K^{\sharp}$ for
some compact group $K\subset U(d-1)$.  In fact, in this work, Katz
proves the required equidistribution for (most) fixed $t$ \emph{in the
  vertical direction} where one looks at characters of $\Fp_{q^m}$ for
fixed $q$ and $m\ra +\infty$ (see~\cite[Th. 7.2, Th. 17.6, Remark
17.7]{katz-mellin}). It seems quite possible that the ``horizontal''
direction we are interested in will also follow from these new
techniques. 
\par
Similarly, it is likely that families of hyper-Kloosterman sums of
certain types exhibit full unitary monodromy, e.g., for any integer
$n\geq 1$, any additive character $\psi\,:\, \Fp_q\ra \Cc^{\times}$,
any multiplicative character $\chi\,:\, \Fp_{q}^{\times}\ra
\Cc^{\times}$, one can consider the sums
$$
\sum_{\stacksum{x_1,\ldots,x_{n}\in\Fp_q^{\times}}{x_1\cdots x_n=1}}{
  \chi(x_1)\psi(x_1+\cdots+x_{n-1}+x_n) },
$$
for which the basic theory (due to Deligne~\cite[\S 7]{sga4h}) shows
that the associated $L$-function (unitarily normalized so the central
point is $s=0$) is
$$
\det(1-q^{-s}\theta_{\psi,\chi}(q))
$$
for some unique conjugacy class $\theta_{\psi,\chi}(q)\in U(n-1)$. One
would then expect (this was suggested by Katz, and is again
potentially in the realm of his recent work~\cite{katz-mellin}) that
if we fix a character $\psi$ for $\Fp_p$ (e.g., $\psi(x)=e(x/p)$) and
define $\psi_m(x)=\psi(\Tr_{\Fp_{p^m}/\Fp_p}(x))$ for all $m\geq 1$,
the sets of conjugacy classes
$$
\{
\theta_{\psi_m,\chi}(p^m)\,\mid\,
\chi\,:\, \Fp_{p^m}^{\times}\ra \Cc^{\times}
\}
$$
become equidistributed in $U(n-1)^{\sharp}$ as $m\ra +\infty$. 
\end{remark}

\subsection{Symplectic symmetry}\label{ssec-symplectic}

A typical example of symplectic symmetry involves families of
$L$-functions of algebraic curves over finite fields. For simplicity,
we will consider one of the simplest ones, but we first start by
proving distribution results for characteristic polynomials of
symplectic matrices, which are of independent interest.  
\par
We first remark that for $A\in USp(2g,\Cc)$, the characteristic
polynomial can be expressed in the form
$$
\det(1-TA)=\prod_{1\leq j\leq g}{(1-e^{i\theta_j}T)(1-e^{-i\theta_j}T)}
$$
for some eigenangles $\theta_j$, $1\leq j\leq g$, and it follows that
$$
\det(1-A)=\prod_{1\leq j\leq g}{|(1-e^{i\theta_j})|^2}\geq 0.
$$
\par
This positivity is reflected in a shift in expectation in the
mod-Gaussian convergence (it also means that the argument is not
an interesting quantity here). We obtain:

\begin{proposition}\label{pr-symplectic}
  For $g\geq 1$, let 
$$
X_g=\log\det(1-T_g)-\demi\log (\tfrac{\pi g}{2}),
$$ 
where $T_g$ is a Haar-distributed random matrix in the unitary
symplectic group $USp(2g,\Cc)$. Then $X_g$ converges in mod-Gaussian
sense with $Q_g(t)=(\log \demi g)t^2$ and limiting function\footnote{\
  The expressions in~\cite[(32), (67)]{ks-l-functions}
  and~\cite[Cor. 4.2]{conrey-farmer} are rather more complicated, but
  of course they are equivalent.}
\begin{equation}\label{eq-symplectic-limiting}
\Phi_{Sp}(t)=
\frac{G(\tfrac{3}{2})}{G(\tfrac{3}{2}+it)}.
\end{equation}
\par
Indeed, we have
$$
\expect(e^{itX_g})=
\exp(-(\log \demi g)t^2/2)\Phi_{Sp}(t)\Bigl(1+O
\Bigl(
\frac{1+|t|^3}{g}
\Bigr)
\Bigr)
$$
for $|t|\leq g^{1/6}$, where the implied constant is absolute.
\end{proposition}

Figure~2 is a graph illustrating the logarithmic growth of
$\tfrac{1}{t^2}\log |\Phi_{Sp}(t)|$.
\begin{figure}[ht]
\centering
\includegraphics[width=4in]{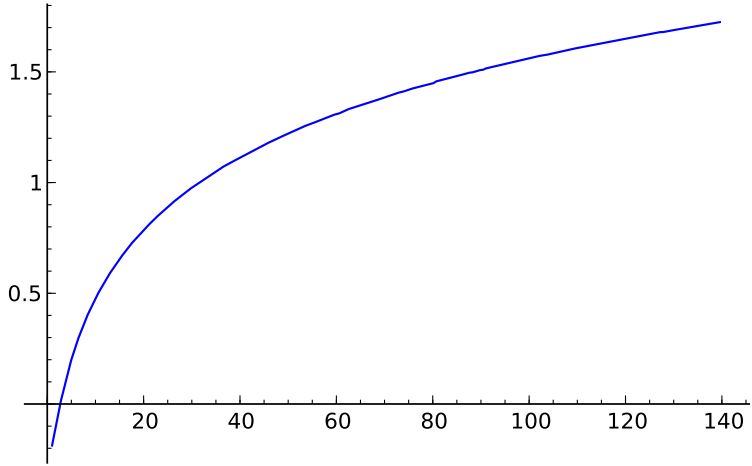}
\caption{Graph of $\tfrac{1}{t^2}\log |\Phi_{Sp}(t)|$, $1\leq t\leq
  140$}
\end{figure}

\begin{proof}
  (Compare~\cite[Prop. 4.9]{jkn})
  Keating-Snaith~\cite[(10)]{ks-l-functions} compute that
$$
\expect(e^{it \log \det(1-T_g)})=
2^{2git}\prod_{j=1}^g{
\frac{\Gamma(1+g+j)\Gamma(\demi + it+j)}
{\Gamma(\demi +j)\Gamma(1+it+g+j)}
},
$$
which, together with the formula~(\ref{eq-ratio-gamma}), gives
$$
\expect(e^{it X_g})=\Bigl(\frac{\pi g}{2}\Bigr)^{-it/2}
\frac{G(\tfrac{3}{2})}{G(\tfrac{3}{2}+it)}
\times 
2^{2git}
\frac{G(\tfrac{3}{2}+it+g)G(2+2g)G(2+it+g)}
{G(\tfrac{3}{2}+g)G(2+g)G(2+it+2g)}.
$$
\par
By applying Proposition~\ref{pr-app}, (3) in the Appendix, we get
$$
\expect(e^{itX_g})=\Bigl(\frac{g}{2}\Bigr)^{-t^2/2}\Phi_{Sp}(t)
\Bigl(1+O\Bigl(\frac{1+|t|^3}{g}\Bigr)\Bigr),
$$
as claimed.
\end{proof}

In particular, the Central Limit Theorem for $\det(1-T_g)$ takes the
form  of the convergence in law
$$
\frac{\log \det(1-T_g)-\demi \log \tfrac{\pi g}{2}}
{(\log( g/2))^{1/2}}\quad\convlaw\quad \text{(standard Gaussian)},
$$
so, for any $a<b$, we have
$$
\proba\Bigl(
\Bigl(\frac{\pi g}{2}\Bigr)^{1/2}
e^{a\sqrt{\log (g/2)}}
<\det(1-T_g)< 
\Bigl(\frac{\pi g}{2}\Bigr)^{1/2}
e^{b\sqrt{\log (g/2)}}
\Bigr)\ra \frac{1}{\sqrt{2\pi}}
\int_a^b{e^{-t^2/2}dt}.
$$
\par
On the other hand, by applying Theorem~\ref{th-2}, as we can according
to the previous proposition, we can control the probability of the
values of $\det(1-T_g)$ in much smaller (dyadic or similar) intervals:

\begin{corollary}
Let $U=]a,b[$ with $a<b$ real numbers. We have
\begin{multline*}
\proba\Bigl(
e^a\Bigl(\frac{\pi g}{2}\Bigr)^{1/2}
<\det(1-T_g)< 
e^b\Bigl(\frac{\pi g}{2}\Bigr)^{1/2}
\Bigr)=
\frac{1}{\sqrt{2\pi  \log \demi g}}
\int_{a}^b{\exp\Bigl(-\frac{t^2}{2\log \demi g}\Bigr)dt}\\
+O\Bigl(\frac{\max(b-a,(b-a)^{-1})}{\log g}+
\frac{\max(1,b-a)}{(\log g)^{1/2+1/29}}
\Bigr)
\end{multline*}
for $g\geq 2$, where the implied constant is absolute. In particular
\begin{multline*}
\proba\Bigl(
\Bigl(\frac{\pi g}{2}\Bigr)^{1/2}
<\det(1-T_g)< 
2\Bigl(\frac{\pi g}{2}\Bigr)^{1/2}
\Bigr)=
\frac{1}{\sqrt{2\pi \log \demi g}}
\int_{0}^{\log 2}{\exp\Bigl(-\frac{t^2}{2\log \demi g}\Bigr)dt}\\
+O\Bigl(\frac{1}{(\log g)^{1/2+1/29}}
\Bigr).
\end{multline*}
\end{corollary}

\begin{proof}
  If $b-a\leq 1$, we apply Theorem~\ref{th-2}, with the constants
  $\mu=1$, $A$ arbitrarily close to $2$, $a$ arbitrarily large and
  $C=1/2$, so that $D$ can be any number with
$$
D>2(1+1+12)=28,
$$
and in particular $D=29$ is valid. If $b-a>1$, we split the interval
$]a,b[$ into $2\lceil b-a\rceil$ intervals of length
$$
\frac{1}{4}\leq
\frac{b-a}{2\lceil
b-a\rceil}\leq 1,
$$
and apply the previous case to the interior of those intervals. Since
the joint distribution of eigenvalues of $T_g$ is absolutely
continuous with respect to Lebesgue measure, the probability of
falling on one of the missing endpoints is zero, and summing over
these intervals gives the result.
\end{proof}

We now deduce an arithmetic corollary, using families of hyperelliptic
curves over finite fields.  For any odd $q$, any integer $g\geq 1$ and
any squarefree monic polynomial $f\in \Fp_q[X]$ of degree $2g+1$, let
$C_f$ be the smooth projective model of the affine hyperelliptic curve
$$
C_f\,:\, y^2=f(x).
$$
\par 
The number of $\Fp_{q^m}$-rational points on $C_f$ satisfies
$$
|C_f(\Fp_{q^m})|= q^m+1-\sum_{x\in
  \Fp_{q^m}}{\chi_2(N_{\Fp_{q^m}/\Fp_q}(f(x)))}=
q^m+1-S_m(\chi_2,f)
$$
where $\chi_2$ is the quadratic character of $\Fp_q^{\times}$ and the
notation is as in Section~\ref{ssec-unitary}. The associated
$L$-function (the numerator of the zeta function) is defined by
$$
L(C_f,s)=L(s,f,\chi_2),
$$
or, in other words, we have
$$
L(C_f,s)=Z(C_f,q^{-s}),\quad\quad
Z(C_f,T)=\exp\Bigl(\sum_{m\geq 1}{\frac{S_m(\chi_2,f)}{m}T^m}\Bigr).
$$
\par
Weil proved that $Z(C_f,T)$ is a polynomial in $\Zz[T]$, of degree
$2g$, all roots of which have modulus $\sqrt{q}$, and which is
\emph{symplectic}: there is a unique conjugacy class $\theta_{f}(q)$
in $USp(2g,\Cc)$ such that
$$
L(C_f,s+\demi)=\det(1-q^{-s}\theta_{f}(q)).
$$
 
\begin{theorem}\label{th-hyperelliptic}
  Let $\mathcal{H}_g(\Fp_q)$ be the set of squarefree, monic, polynomials of
  degree $2g+1$ in $\Fp_q[X]$. Fix a non-empty open interval
  $]\alpha,\beta[\subset ]0,+\infty[$. For all $g$ large enough, we
  have
\begin{equation}\label{eq-symplectic}
\liminf_{q\ra +\infty}{\frac{1}{|\mathcal{H}_g(\Fp_q)|}
\Bigl|\Bigl\{
f\in \mathcal{H}_g(\Fp_p)\,\mid\, 
\frac{L(C_f,1/2)}{\sqrt{\pi g/2}}\in ]\alpha,\beta[
\Bigr\}\Bigr|
}
\gg \frac{1}{\sqrt{\log g}}.
\end{equation}
\end{theorem}

(Note that this is in fact a very weak version of what we can prove).

\begin{proof}
  Let first $\mathcal{H}_g^*(\Fp_q)$ be the set of $f\in
  \mathcal{H}_g(\Fp_q)$ for which
  $L(C_f,1/2)\not=0$. In~\cite[Prop. 4.9]{jkn}, we showed, using the
  relevant equidistribution computation in~\cite[10.8.2]{katzsarnak}
  that the (real-valued) random variables
$$
L_g=\log\det(1-\theta_{F}(q))-\demi \log (\tfrac{\pi g}{2}),
$$
on $\mathcal{H}_g^*(\Fp_q)$ (with counting measure) converges in law to
$$
X_g=\log\det(1-T_g)-\demi\log (\tfrac{\pi g}{2}),
$$
where $T_g$ is a random matrix in the unitary symplectic group
$USp(2g,\Cc)$, distributed according to Haar measure. The previous
proposition shows that Theorem~\ref{th-2} is applicable to $X_g$ with
covariance $Q_g(t)=(\log \demi g)t^2$ and limiting function
$\Phi_{Sp}(t)$. Letting $q\ra +\infty$ as in the previous section, we
get
$$
\liminf_{q\ra +\infty}
\frac{1}{|\mathcal{H}_g^*(\Fp_q)|}
\Bigl|\Bigl\{
f\in \mathcal{H}_g^*(\Fp_p)\,\mid\, 
\log L(C_f,1/2)-\demi \log (\tfrac{\pi g}{2})
\in ]\alpha,\beta[\Bigr\}\Bigr|
\gg \frac{1}{\sqrt{\log g}}
$$
for $g$ large enough. Since
$$
|\mathcal{H}_g^*(\Fp_q)|=|\mathcal{H}_g(\Fp_q)|(1+o(1))=q^{2g+1}(1+o(1))
$$
for fixed $g$ and $q\ra +\infty$ (by an easy application of the
equidistribution, see~\cite[Prop. 4.9]{jkn}), we get the result stated
by exponentiating.
\end{proof}

\begin{remark}
  The lower bound~(\ref{eq-symplectic}) is good enough to combine with
  various other statements proving arithmetic properties of
  $L$-functions which hold for ``most'' hyperelliptic curves.  For
  instance, from~\cite[Prop. 1.1]{kow-relations} (adapted
  straightforwardly to all hyperelliptic curves instead of special
  one-parameter families), it follows that if we denote by
  $\tilde{\mathcal{H}}_g(\Fp_q)$ the set of $f\in
  \mathcal{H}_g(\Fp_q)$ such that the eigenvalues of $\theta_f(q)$
  satisfy no non-trivial multiplicative relation,\footnote{\ An
    analogue of the hypothetical statement of $\Qq$-linear
    independence of the ordinates of zeros of the Riemann zeta
    function; non-trivial refers to a relation that can not be deduced
    from the fact that, if $e^{i\theta}$ is an eigenvalue, so is its
    inverse $e^{-i\theta}$.}  then we have
$$
|\{f\in \mathcal{H}_g(\Fp_q)\,\mid\, f\notin
\tilde{\mathcal{H}}_g(\Fp_q)\}|\ll_g q^{1-\gamma}
$$
for some $\gamma=\gamma(g)>0$, and hence we get
$$
\liminf_{q\ra +\infty}{\frac{1}{|\mathcal{H}_g(\Fp_q)|} 
\Bigl|\Bigl\{ 
f\in \tilde{\mathcal{H}}_g(\Fp_p)\,\mid\, 
\frac{L(C_f,1/2)}{\sqrt{\pi g/2}}\in ]\alpha,\beta[ 
\Bigr\}\Bigr| } \gg \frac{1}{\sqrt{\log g}}
$$
for $g$ large enough.
\end{remark}


\subsection{Orthogonal symmetry}

Orthogonal symmetry, in number theory, features pro\-minently in
families of elliptic curves. In contrast with symplectic groups, there
are a number of ``flavors'' involved, due to the ``functional equation''
$$
T^N\det(1-T^{-1}A)=\det(-A)\det(1-TA)
$$
for an orthogonal matrix $A\in O(N,\Rr)$ (the standard maximal compact
subgroup of the orthogonal group $O(N,\Cc)$), which implies that
$\det(1-A)$ is zero for ``trivial'' reasons if $N$ is even and
$\det(A)=-1$ or $N$ is odd and $\det(A)=1$. When this happens, it is
of great interest to investigate the distribution of the first
derivative at $T=1$ of the reversed characteristic polynomial. For
simplicity, however, we restrict our attention here to $N$ even and
matrices with determinant $1$, i.e., to the subgroup $SO(2N,\Rr)$ of
$O(2N,\Rr)$, where $N\geq 1$. In that case, it is also true that
eigenangles come in pairs of inverses, and therefore we have
$\det(1-A)\geq 0$.
\par
As in the previous section, we start with random matrix computations.

\begin{proposition}\label{pr-orthogonal}
  For $N\geq 1$, let 
$$
X_N=\log\det(1-T_N)-\demi\log (\tfrac{8\pi }{N}),
$$
where $T_N$ is a Haar-distributed random matrix in the special
orthogonal group $SO(2N,\Rr)$. Then $X_N$ converges in mod-Gaussian
sense with $Q_N(t)=(\log \demi N)t^2$ and limiting function
  \begin{equation}\label{eq-orthogonal-limiting}
\Phi_{SO}(t)=
\frac{G(\tfrac{1}{2})}{G(\tfrac{1}{2}+it)}.
\end{equation}
\par
Indeed, we have
$$
\expect(e^{itX_N})=
\exp(-(\log \demi N)t^2/2)\Phi_{SO}(t)\Bigl(1+O
\Bigl(
\frac{1+|t|^3}{N}
\Bigr)
\Bigr)
$$
for $|t|\leq N^{1/6}$, where the implied constant is absolute.
\end{proposition}

\begin{proof}
Using~\cite[(56)]{ks-l-functions} and~(\ref{eq-ratio-gamma}), we get
\begin{align*}
\expect(e^{it \log \det(1-T_N)})&=
2^{2Nit}\prod_{j=1}^N{
\frac{\Gamma(N+j-1)\Gamma(it+j-\demi)}
{\Gamma(j-\demi)\Gamma(it+N+j-1)}
},
\\
&=
\frac{G(\tfrac{1}{2})}{G(\tfrac{1}{2}+it)}
\times 
2^{2Nit}
\frac{G(\tfrac{1}{2}+it+N)G(2N)G(it+N)}
{G(\tfrac{1}{2}+N)G(N)G(it+2N)},
\end{align*}
and by applying Proposition~\ref{pr-app}, (4) in the Appendix, we get
the desired formula
$$
\expect(e^{itX_N})=\Bigl(\frac{N}{2}\Bigr)^{-t^2/2}\Phi_{SO}(t)
\Bigl(1+O\Bigl(\frac{1+|t|^3}{N}\Bigr)\Bigr).
$$
\end{proof}

\begin{remark}
  If we compare with the symplectic case, we observe the (already
  well-established) phenomenon that the value $\det(1-A)$, for $A\in
  SO(2N,\Rr)$ tend to be small, whereas they tend to be large for
  symplectic matrices in $USp(2g,\Cc)$.
\end{remark}

Our arithmetic corollary is based on families of quadratic twists of
elliptic curves over function fields, and we select a specific example
for concreteness (see~\cite[\S 4]{quad-twists}); the basic theory,
which we illustrate here, is again due to Katz~\cite{katz-twists}.
\par
For any odd prime power $q\geq 3$, any integer $N\geq 1$, we consider
the elliptic curves over the functional field $\Fp_q(T)$ given by the
Weierstrass equations
$$
\mathcal{E}_z\,:\, 
Y^2=(T^N-NT-1-z)X(X+1)(X+T),
$$
where $z\in\Fp_q$ is a parameter such that $z$ is not a critical value
of $T^N-NT-1$.
\par
Katz proved that the associated $L$-function (which is now defined by
the ``standard'' Euler product over primes in $\Fp_q[T]$, with
suitable ramified factors) is of the form
$$
L(\mathcal{E}_z,s+1)=\det(1-\theta_{z}(q)q^{-s})
$$
where $\theta_{z}(q)$ is a unique conjugacy class in $O(2N,\Rr)$.

 
\begin{theorem}\label{th-orthogonal}
  Fix a non-empty open interval $]\alpha,\beta[\subset
  ]0,+\infty[$. For all $N$ large enough, we have
$$
\liminf_{\stacksum{p\ra +\infty}{(p-1,N-1)=1}}{\frac{1}{p}
  \Bigl|\Bigl\{ z\in \Fp_p\,\mid\, \Bigl(\frac{N}{8\pi}\Bigr)^{1/2}
  L(\mathcal{E}_z,1/2)\in ]\alpha,\beta[ \Bigr\}\Bigr|} \gg
\frac{1}{\sqrt{\log N}}.
$$
\end{theorem}

\begin{proof}
  As recalled in~\cite[Cor. 4.4 and before]{quad-twists}, for all
  $N\geq 146$ and primes $p$ with $p\nmid N(N-1)(N+1)$ and
  $(p-1,N-1)=1$, the conjugacy classes $\theta_{z}(p)$, for
  $z\in\Fp_p$ not a critical value, become equidistributed in
  $O(2N,\Rr)^{\sharp}$ for the image of Haar measure (precisely, this
  is stated for the ``vertical direction'' where $p$ is fixed and
  finite fields of characteristic $p$ and increasing degree are used;
  however, because the parameter variety is a curve with $N+1$ points
  at infinity and the relevant sheaf is tame, we can recover the
  horizontal statement as in the proof of Theorem~\ref{th-katz}). In
  particular, there is a subset $V_p\subset \Fp_p$ with $|V_p|\sim
  p/2$ where $\det(\theta_{z}(p))=1$ and those restricted conjugacy
  classes become equidistributed in $SO(2N,\Rr)^{\sharp}$. Hence, for
  $N$ large enough, we get
\begin{multline*}
\liminf_{\stacksum{p\ra +\infty}{(p-1,N-1)=1}}
\frac{1}{|V_p|} \Bigl|\Bigl\{
  z\in V_p\,\mid\, \Bigl(\frac{N}{8\pi}\Bigr)^{1/2}
  L(\mathcal{E}_z,1/2)\in ]\alpha,\beta[ \Bigr\}\Bigr|
\\\geq \mu_{SO(2N,\Rr)}(\{A\,\mid\, 
\log \det(1-A)\in]\log \alpha,\log\beta[\})
\gg \frac{1}{\sqrt{\log N}},
\end{multline*}
as desired.
\end{proof}

\begin{remark}
  Obviously, this result (or its generalizations to other families of
  quadratic twists over function fields) has interesting consequences
  concerning the problem of the distribution of the order of
  Tate-Shafarevich groups of the associated elliptic curves, through
  the Birch and Swinnerton-Dyer conjecture (which is known to be valid
  in its strong form for many elliptic curves over function fields
  over a finite field with analytic rank $0$ or $1$). We hope to come
  back to this question, and its conjectural analogue over number
  fields, in another work.
\end{remark}


\section*{Appendix: estimates for the Barnes function}

We present in this appendix some uniform analytic estimate for the
Barnes function, which are needed to verify the strong convergence
assumption~(\ref{eq-strong-convergence}) for sequences of random
matrices in compact classical groups. Note that we did not try to
optimize the results.

\begin{proposition}\label{pr-app}
  \emph{(1)} For all $z\in\Cc$ and $n\geq 1$ with $|z|\leq n^{1/6}$,
  we have
\begin{equation}\label{eq-power-ratio}
\frac{G(1+z+n)}{G(1+n)}=(2\pi)^{z/2}e^{-(n+1)z}(1+n)^{z^2/2+nz}
\Bigl(1+O\Bigl(\frac{z^2+z^3}{n}\Bigr)\Bigr).
\end{equation}
\par
\emph{(2)} For all $N\geq 1$ and all $t=(t_1,t_2)\in\Rr^2$ with $\|t\|\leq
N^{1/6}$ we have
$$
\frac{G(1+it_1+N)G(1+N)}{G(1+\tfrac{it_1-t_2}{2}+N)
G(1+\tfrac{it_1+t_2}{2}+N)}
=N^{-(t_1^2+t_2)/4}\Bigl(1+
O\Bigl(\frac{1+\|t\|^3}{N}\Bigr)\Bigr).
$$
\par
\emph{(3)} For all $g\geq 1$ and all $t\in\Rr$ with
$|t|\leq g^{1/6}$ we have
$$
2^{2git}
\frac{G(\tfrac{3}{2}+it+g)G(2+2g)G(2+it+g)}
{G(\tfrac{3}{2}+g)G(2+g)G(2+it+2g)}=
\Bigl(\frac{g}{2}\Bigr)^{-t^2/2}
\Bigl(\sqrt{\frac{\pi g}{2}}\Bigr)^{it}\Bigl(1+
O\Bigl(\frac{1+|t|^3}{g}\Bigr)\Bigr).
$$
\par
\emph{(4)} For all $N\geq 1$ and all $t\in\Rr$ with
$|t|\leq N^{1/6}$ we have
$$
2^{2Nit} \frac{G(\tfrac{1}{2}+it+N)G(2N)G(it+N)}
{G(\tfrac{1}{2}+N)G(g)G(it+2N)}= \Bigl(\frac{N}{2}\Bigr)^{-t^2/2}
\Bigl(\sqrt{\frac{8\pi}{N}}\Bigr)^{it}\Bigl(1+
O\Bigl(\frac{1+|t|^3}{N}\Bigr)\Bigr).
$$
\par
In all estimates, the implied constants are absolute.
\end{proposition}


\begin{proof}
  One can use the asymptotic expansions in~\cite{fl}, but we follow
  instead the nice arrangement of the Barnes function
  in~\cite[Cor. 3.2]{e-s}, which leads to a quicker and cleaner proof.
\par
(1) First, the ratio of Barnes function is well-defined since $n\geq
1$. We now use the formula
\begin{equation}\label{eq-formula}
\frac{G(1+z+n)}{G(1+n)}=(2\pi)^{z/2}e^{-(n+1)z}(1+n)^{z^2/2+nz}S_n(z),
\end{equation}
where
$$
S_n(z)=e^{-z(z-1)/2} \prod_{k\geq n+1}{
  \Bigl(1+\frac{z}{k}\Bigr)^{k-n}\Bigl(1+\frac{1}{k}\Bigr)^{z^2/2+nz}
  e^{-z} },
$$
which is valid for $z\in\Cc$, $n\geq 1$ (see~\cite[p. 241]{e-s}).
\par
If we expand the logarithm, defined using the Taylor expansion of
$\log(1+w)$ at the origin, we have
$$
\log(1+w)=w-\frac{w^2}{2}+O(w^{-3})
$$
for $|w|\leq 1/2$, with an absolute implied constant. Hence we obtain
\begin{align*}
\log S_n(z)&=
-\frac{z(z-1)}{2}+
\sum_{k>n}{\Bigl(k\log \Bigl(1+\frac{z}{k}\Bigr)+\frac{z^2}{2}
\log\Bigl(1+\frac{1}{k}\Bigr)-z\Bigr)}\\
&\quad\quad +\sum_{k>n}{n\Bigl(z\log\Bigl(1+\frac{1}{k}\Bigr)
-\log\Bigl(1+\frac{z}{k}\Bigr)\Bigr)}
\\
&=-\frac{z(z-1)}{2}-\frac{z^2}{2}\sum_{k>n}{\frac{1}{k^2}}
+n\frac{z(z-1)}{2}\sum_{k>n}{\frac{1}{k^2}}
+O\Bigl(\sum_{k>n}{\Bigl(\frac{z^2}{k^3}+\frac{z^3}{k^2}\Bigr)}\Bigr)
\end{align*}
with an absolute implied constant, for $n\geq 1$ and $|z|\leq n/2$,
hence for $|z|\leq n^{1/6}$ we get
$$
\log S_n(z)=O\Bigl(\frac{z^2+z^3}{n}\Bigr),
$$
since
$$
\sum_{k>n}{\frac{1}{k^2}}=\frac{1}{n}+O\Bigl(\frac{1}{n^2}\Bigr),
\quad
\sum_{k>n}{\frac{1}{k^3}}=\frac{1}{2n^2}+O\Bigl(\frac{1}{n^3}\Bigr),
$$
for $n\geq 1$, with absolute implied constants. Hence, we have
$$
S_n(z)=1+O\Bigl(\frac{z^2+z^3}{n}\Bigr)
$$
for $|z|\leq n^{1/6}$, for some absolute implied constant, and we get
the stated formula
$$
\frac{G(1+z+n)}{G(1+n)}=(2\pi)^{z/2}e^{-(n+1)z}(1+n)^{z^2/2+nz}
\Bigl(1+O\Bigl(\frac{z^2+z^3}{n}\Bigr)\Bigr).
$$
\par
(2) Note first that the conditions $N\geq 1$ and $\|t\|\leq N^{1/6}$
ensure that the values of the Barnes function in the denominator are
non-zero. Next, let $u=(it_1-t_2)/2$, $v=(it_1+t_2)/2$; we can express
the ratio of Barnes function as
$$
\frac{G(1+u+v+N)G(1+N)}{G(1+u+N)
G(1+v+N)}
=\frac{G(1+u+v+N)}{G(1+N)}
\frac{G(1+N)}{G(1+u+N)}
\frac{G(1+N)}{G(1+v+N)},
$$
and $\|t\|\leq N^{1/6}$ gives $|u|$, $|v|\leq N^{1/6}$, allowing us to
apply~(\ref{eq-power-ratio}) three times. The exponential terms cancel
out, leading to
$$
\frac{G(1+u+v+N)G(1+N)}{G(1+u+N)
G(1+v+N)}
=(1+N)^{-\|t\|^2/4}\Bigl(1+O\Bigl(\frac{|t|^2+|t|^3}{N}\Bigr)
\Bigr)
$$
which gives the first part of the proposition.
\par
(3) We use a similar computation, applying~(\ref{eq-power-ratio}) six
times with the parameters $(n,z)$ 
$$
(2g,1),\  (g,1+it),\  (g,\demi+it),\ 
(g,1),\  (2g,1+it),\  (g,\demi),
$$
leading, after an easy calculation, to a main term
$$
(2\pi)^{it/2}e^{-it}(1+g)^{-t^2+3it/2+2igt}
(1+2g)^{t^2/2-it-2igt}
$$
for the ratio of Barnes functions, and some further computation leads
to the stated result (each parameter $z$ has $|z|\leq |t|+1$, so the
error term is also as given).
\par
(4) We argue exactly as in the previous case, with parameters $(n,z)$
given now by
\begin{gather*}
(2N-1,0),\  (N-1,0),\  (N-1,it+\demi),\ 
 (N-1,it),\  (2N-1,it),\  (N-1,\demi),
\end{gather*}
and we get a main term
$$
(2\pi)^{it/2}N^{-t^2-3it/2+2iNt}(2N)^{t^2/2+it-2itN}=
2^{t^2/2+it-2itN}(2\pi)^{it/2}N^{-t^2/2-it/2},
$$
which leads to the conclusion.
\end{proof}

\end{document}